\documentclass[12pt,reqno]{amsart}
\usepackage{amsmath,amssymb,amscd,fancyhdr,mathrsfs}
\usepackage[margin=0.8in]{geometry}
\usepackage{graphicx}
\usepackage{color}
\usepackage{bm} % for bold italic letters
\usepackage{enumitem} % for enumerate environment
\allowdisplaybreaks
\usepackage{esint}

\newcommand{\R}{{\mathbb R}}
\newcommand{\D}{{\mathbb D}}

\newcommand{\blue}{\textcolor{black}}

\newcommand{\pa}{\partial}
\newcommand{\na}{\nabla}
\newcommand{\T}{\mathbb T}
\newcommand{\di}{\nabla\cdot}
\newcommand{\eps}{\varepsilon}
\newcommand{\intO}{\int_{\Omega}}

\newcommand{\ue}{\uu^{\eps}}

\newcommand{\dH}{d\mathcal{H}^{d-1}}

\newcommand{\ox}{\omega}
\newcommand{\ii}{\int_{\Omega_{\eps_2}}}
\newcommand{\thii}{\frac{1}{\eps_3}\int_{\eps_1}^{\eps_1+\eps_3}\int_\tau^t\ii}
\newcommand{\thiip}{\frac{1}{\eps_3}\int_{\eps_1}^{\eps_1+\eps_3}\int_\tau^t\int_{\pa\Omega_{\eps_2}}}
\newcommand{\vre}{\vr^{\eps}}

\newcommand{\GG}{\mathcal{G}}
\newcommand{\HH}{\mathcal{P}}

\newcommand{\dxte}{dxdsd\eps_2}
\newcommand{\dHte}{d\mathcal{H}^{d-1}(\theta)dsd\eps_2}
\newcommand{\intTd}{\int_{\T^d}}

\newtheorem{theorem}{Theorem}[section]
\newtheorem{lemma}[theorem]{Lemma}

\newtheorem{remark}[theorem]{Remark}

\newtheorem{definition}{Definition}[section]

\newcommand{\vr}{\varrho}
\newcommand{\uu}{\mathbf{u}}

\usepackage[colorlinks=true,citecolor=magenta,linkcolor=magenta]{hyperref}
\begin{document}

\title[Energy equalities for compressible Navier-Stokes equations]{Energy equalities for compressible Navier-Stokes equations}

\author[Q-H. Nguyen]{Quoc-Hung Nguyen}
\author[P-T. Nguyen]{Phuoc-Tai Nguyen}
\author[B. Q. Tang]{Bao Quoc Tang}

\address{Quoc-Hung Nguyen \hfill\break
Department of Mathematics, New York University Abu Dhabi, Abu Dhabi, United Arab Emirates}
\email{qn2@nyu.edu} 

\address{Phuoc-Tai Nguyen \hfill\break
	Faculty of Science, Department of Mathematics and Statistics, Masaryk University, 61137 Brno,
	Czech Republic}
\email{ptnguyen@math.muni.cz, nguyenphuoctai.hcmup@gmail.com} 

\address{Bao Quoc Tang \hfill\break
	Institute of Mathematics and Scientific Computing, University of Graz, Heinrichstrasse 36, 8010 Graz, Austria}
\email{quoc.tang@uni-graz.at, baotangquoc@gmail.com} 

\date{\today}

\thanks{}

\begin{abstract}
	The energy equalities of compressible Navier-Stokes equations with general pressure law and degenerate viscosities are studied. By using a unified approach, we give sufficient conditions on the regularity of weak solutions for these equalities to hold. The method of proof is suitable for the case of periodic as well as homogeneous Dirichlet boundary conditions. In particular, by a careful analysis using the homogeneous Dirichlet boundary condition, no boundary layer assumptions are required when dealing with bounded domains with boundary.
\end{abstract}


% \paragraph{AMS classification:}

\maketitle

%\tableofcontents

%%%%%%%%%%%%%%%%%%%%%%%%%%%%%%%%%%%%%%%%%%%%%%%%%%%%%%%%%%%%%%%%%%%%%%%%%%%%%%%
\tableofcontents

\section{Introduction and Main Results}
Let $d=2,3$ and $\Omega$ be either the torus $\T^d$ or a bounded domain in $\mathbb R^d$ with $C^2$ boundary $\partial\Omega$. 
This paper studies the energy equalities for compressible Navier-Stokes equation with \blue{degenerate viscosities and general pressure law}
\begin{equation}\label{gNS}\tag{gcNS}
	\begin{cases}
		\partial_t \vr + \di(\vr \uu) = 0, &\text{ in } \Omega \times (0,T),\\
		\partial_t (\vr \uu) + \di(\vr \uu\otimes \uu) + \blue{\nabla (p(\vr)) -\di(\nu(\vr)\D \uu) - \nabla(\mu(\vr)\di \uu)} = 0, &\text{ in } \Omega \times (0,T).
		%\uu = 0, &\text{ on } \partial\Omega\times (0,T),\\
		%(\vr\uu)(x,0)  = \vr_0(x)\uu_0(x), &\text{ in } \Omega,\\
		%\vr(x,0) = \vr_0(x), &\text{ in } \Omega,
	\end{cases}
\end{equation}
and the inhomogeneous incompressible Navier-Stokes equation
\begin{equation}\label{icNS}\tag{icNS}
\begin{cases}
\partial_t \vr + \di(\vr \uu) = 0, &\text{ in } \Omega \times (0,T),\\
\partial_t (\vr \uu) + \di(\vr \uu\otimes \uu) + \nabla P -\blue{\di (\nu(\varrho) \D \uu)} = 0, &\text{ in } \Omega \times (0,T),\\
\na\cdot \uu = 0, &\text{ in } \Omega\times (0,T),\\
%\uu = 0, &\text{ on } \partial\Omega\times (0,T),\\
%(\vr\uu)(x,0)  = \vr_0(x)\uu_0(x), &\text{ in } \Omega,\\
%\vr(x,0) = \vr_0(x), &\text{ in } \Omega,
\end{cases}
\end{equation}
with initial data 
\begin{equation} \label{initial}   \left\{ \begin{aligned}
(\vr\uu)(x,0) &= \vr_0(x)\uu_0(x), \quad x\in \Omega, \\
\vr(x,0) &= \vr_0(x), \quad x \in \Omega,
\end{aligned} \right.  \end{equation}
and homogeneous Dirichlet boundary condition 
\begin{equation} \label{boundary}  \uu = 0 \quad \text{on} \quad \pa\Omega\times(0,T).
\end{equation}
Here $T>0$ is the time horizon, $\varrho$ and $\uu$ denote the density and the velocity of the fluid, respectively, $\D=\frac{1}{2}[\nabla \uu + \nabla^\top \uu]$ stands for the strain tensor, and in the case of \eqref{icNS}, $P$ is the scalar pressure. \blue{The general pressure law $p(\vr)$, the viscosity coefficients $\nu(\vr)$ and $\mu(\vr)$ satisfy some conditions which will be specified later.} Naturally, the boundary condition $\uu = 0$ on $\partial\Omega\times(0,T)$ is only imposed in the case of bounded domain.

It is remarked that when $\nu(\vr) = \nu_0 \vr$ and $\mu(\vr) = 0$ for some constant $\nu_0>0$, and $p(\vr) = \vr^\gamma$ with $\gamma>1$, \eqref{gNS} reduces to the well-known isentropic compressible Navier-Stokes equation with linear degenerate viscosity and $\gamma$-pressure law
%We consider the compressible Navier-Stokes equation 
\begin{equation}\label{NS}\tag{cNSd}
\begin{cases}
\partial_t \vr + \di(\vr \uu) = 0, &\text{ in } \Omega \times (0,T),\\
\partial_t (\vr \uu) + \di(\vr \uu\otimes \uu) + \nabla(\vr^\gamma) -\nu_0 \nabla \cdot (\varrho \D \uu) = 0, &\text{ in } \Omega \times (0,T),\\
%\uu = 0, &\text{ on } \partial\Omega\times (0,T),\\
%(\vr\uu)(x,0)  = \vr_0(x)\uu_0(x), &\text{ in } \Omega,\\
%\vr(x,0) = \vr_0(x), &\text{ in } \Omega,
\end{cases}
\end{equation}
and when $\nu(\vr) \equiv \nu_0$ and $\mu(\vr) \equiv \lambda_0$, $p(\vr) = \vr^\gamma$ with $\gamma>1$, \eqref{gNS} reduces to
\begin{equation}\label{cNS-normal}\tag{cNS}
\begin{cases}
\partial_t \vr + \di(\vr \uu) = 0, &\text{ in } \Omega \times (0,T),\\
\partial_t (\vr \uu) + \di(\vr \uu\otimes \uu) + \nabla(\vr^\gamma) -\nu_0 \Delta \uu -\lambda_0 \nabla(\di \uu) = 0, &\text{ in } \Omega \times (0,T).
%\uu = 0, &\text{ on } \partial\Omega\times (0,T),\\
%(\vr\uu)(x,0)  = \vr_0(x)\uu_0(x), &\text{ in } \Omega,\\
%\vr(x,0) = \vr_0(x), &\text{ in } \Omega,
\end{cases}
\end{equation}

\medskip
Energy conservation is one aspect of the Onsager's conjecture which was announced in his celebrated paper on statistical hydrodynamics \cite{Ons49}. More precisely, Onsager \cite{Ons49} conjectured that, in the context of homogeneous incompressible Euler equations, kinetic energy is globally conserved for H\"older continuous solutions with the exponent greater than $1/3$, while energy dissipation phenomenon occurs for H\"older continuous solutions with the exponent less than $1/3$. The `positive' part of the conjecture was first proved by by Eyink \cite{Eyi94} and Constantin-E-Titi \cite{CET94} for the whole space $\R^d$ or periodic boundary conditions, i.e. $\T^d$. Significant contributions in the case of domains with boundary were recently achieved in \cite{BT18,BTW,DN18,NN18} for homogeneous incompressible Euler equations and in \cite{NNT18} for inhomogeneous incompressible and compressible isentropic Euler equations. The other direction of the Onsager's conjecture was initiated in the  groundbreaking paper of Scheffer \cite{Sch93}, then has reached to its full flowering with a work of Shnirelman \cite{Shn97}, a series of celebrated works of De Lellis and Sz\'ekelyhidi \cite{DS12, DS13, DS14, BDIS15}, and recently was settled by Isett in \cite{Ise18a, Ise18} and by Buckmaster et al. in \cite{BDSV18}. 

\medskip
Roughly speaking, Onsager's conjecture (for Euler equations) addresses the question how much regularities needed for a weak solution to conserve energy. In the context of classical incompressible Navier-Stokes equations, since global regularity in three dimensions has been a long standing open problem, it is natural to ask how much regularities needed for a weak solution to satisfy the {\it energy equality} (rather than energy conservation in the context of Euler equations). This topic of research has been studied already in the sixties starting with the work of Serrin \cite{Ser62} where he asserted that the energy equality must hold if $\uu \in L^s(0,T;L^q(\T^d))$ with $\frac{2}{s}+\frac{d}{q} \leq 1$. Later Shinbrot \cite{Shi74} assumed a condition which is independent of dimensions, more precisely $\uu \in L^s(0,T;L^q(\T^d))$ with $\frac{1}{s}+\frac{1}{q} \leq \frac 12$ and $q \geq 4$. By using a new approach based on a lemma introduced by Lions, C. Yu \cite{Yu16} obtained the same result as in \cite{Shi74}. All these results deal with either $\Omega = \R^d$ or $\Omega = \T^d$. Recently, C. Yu \cite{Yu17-2} considered the case where $\Omega$ is a bounded domain with $C^2$ boundary and proved the energy equality under additional condition $\uu \in L^s(0,T;B_s^{\alpha,\infty}(\Omega))$ (with $s>2$ and $\frac{1}{2}+\frac{1}{s}<\alpha<1$) which enables to deal with the boundary effects. Here $B_s^{\alpha,\infty}(\Omega)$ denotes the Besov space.

\medskip
Much less is known concerning energy equalities for compressible or inhomogeneous incompressible Navier-Stokes equations. Recent results for \eqref{NS} and \eqref{cNS-normal} in $\T^d$ were carried out by C. Yu \cite{Yu17} and by Akramov et al. \cite{ADSW}. More precisely, in \cite{Yu17}, Yu gave sufficient conditions on the regularity of the density $\varrho$ and the velocity $\uu$ for the validity of the energy equality. The framework employed in \cite{Yu17} is remarkably different from that in the incompressible case (see \cite{Yu16,Yu17-2}). In particular, solutions defined in \cite[Definition 1]{Yu17} are required to satisfy a set of regularities which allow to deduce the continuity of $(\sqrt{\varrho}\uu)(t)$ in the strong topology at $t=0$.  The existence of this kind of solutions is guaranteed by \cite{VY16} and \cite{LV18}. 

\medskip
Motivated by the aforementioned works, we present in this paper a unified approach to show energy equalities for the compressible \eqref{gNS} and the inhomogeneous incompressible Navier-Stokes equations \eqref{icNS}, \blue{in which we consider general pressure law and possible degenerate viscosities (see e.g. \cite{CP10, FGGW17, XZ})}. The main idea is inspired by our recent work for compressible Euler equations \cite{NNT18}, which is different from the method employed in e.g. \cite{Yu16, Yu17, Yu17-2}. In particular, we use the test function  $(\varrho^\varepsilon)^{-1}(\varrho \uu)^\varepsilon$ instead of $\uu^\varepsilon$, where {\it the convolution is only taken in spatial variables}. \blue{The choice of this test function allows to avoid time regularity for the density and the velocity, and to obtain mild conditions which are in fact weaker than previous works\footnote{Admittedly, our approach seems not directly applicable to the case with vacuum. In this case, it seems that more time regularity on the density or the velocity must be imposed.}}. We remark that this idea was also used for density-dependent Navier-Stokes equations  \cite{LS16} (see also \cite[Remark 4.2]{ FGGW17} for the Euler equations). % For instance, when $d = 3$ we assume only $\uu \in L^4(\T^d\times (0,T))$ while \cite{Yu17} required $\uu \in L^s(0,T;L^q(\T^d))$ where $\frac 1s + \frac 1q \leq \frac{5}{12}$ and $q\geq 6$ (see more discussion after Theorem \ref{cNS-torus}). 
Moreover, by carefully using the Dirichlet boundary conditions (in the case of bounded domains), we show that no additional regularities near the boundary on the velocity $\uu$ are required.

%{\color{red} (References about the degenerate viscosity and general pressure.)  }

\medskip

%the present paper, we aim to derive the global kinetic energy conservation for \eqref{NS}, \eqref{cNS-normal} and \eqref{icNS}. 

Before stating the main results, we introduce the definition of weak solutions.

\begin{definition} \label{solution} A couple $(\varrho,\uu)$ is called a weak solution to \eqref{gNS} with initial data \eqref{initial} if 
	\begin{itemize}
	\item[(i)] \begin{align}\label{eqdef1}
	\int_0^T \int_{\Omega} (\varrho \partial_t \varphi  + \varrho \uu\cdot \nabla \varphi ) dxdt = 0 
	\end{align}
	for every test function $\varphi \in C_0^\infty(\Omega \times (0,T))$.
	
	\item[(ii)] \begin{align}\label{eqdef2}
	\int_0^T \int_{\Omega} (\varrho \uu \cdot \partial_t \psi + \varrho \uu \otimes \uu : \nabla \psi + p(\vr) \nabla \cdot \psi +\nu(\vr) \D \uu : \nabla \psi + \mu(\vr)(\di\uu) (\di \psi)) dxdt = 0 
	\end{align}
	for every test vector field $\psi \in C_0^\infty(\Omega \times (0,T))^d$.	
	
	\item[(iii)] $\varrho(\cdot,t) \rightharpoonup \varrho_0$ in ${\mathcal D}'(\Omega)$ as $t \to 0$, i.e. 
	\begin{equation} \label{ini1}
	\lim_{t \to 0}\int_{\Omega} \varrho(x,t) \varphi(x)dx =  \int_{\Omega} \varrho_0(x) \varphi(x)dx 
	\end{equation}
	for every test function $\varphi \in C_0^\infty(\Omega)$. 
	
	\item[(iv)] $(\varrho \uu)(\cdot,t) \rightharpoonup \varrho_0 \uu_0$ in ${\mathcal D}'(\Omega)$ as $t \to 0$, i.e. 
	\begin{equation} \label{ini2}
	\lim_{t \to 0}\int_{\Omega}  (\varrho \uu)(x,t) \psi(x)dx =  \int_{\Omega} (\varrho_0 \uu_0)(x) \psi(x)dx 
	\end{equation}
	for every test vector field $\psi \in C_0^\infty(\Omega)^d$. 
\end{itemize}

\medskip
Weak solutions to  \eqref{icNS} can be defined similarly.
\end{definition} 

\begin{remark}\label{lowregdef}
The notion of weak solutions defined in Definition \ref{solution} requires only modest regularities of the density and the velocity, e.g. for \eqref{eqdef1} and \eqref{eqdef2} to make sense, one only needs $\vr, \vr\uu, \vr\uu\otimes\uu, p(\vr), \nu(\vr) \mathbb D\uu, \mu(\vr)\di\uu \in L_{loc}^1(\Omega \times (0,T))$. If one were to show energy equality for weak solutions defined in \cite{VY16} (or \cite{LV18}), as it was done in \cite{Yu17}, more regularities on these solutions are already given, see \cite[Theorem 1.2]{VY16} or \cite[Definition 1.1]{Yu17}.
\end{remark}
%
%For convenience, we define the following constant
%\begin{equation}\label{alphabeta}
%	\alpha = \min\{\gamma, 2\}.
%\end{equation}

\blue{Throughout this paper, we assume the following conditions on the viscosities and the pressure law:
	\begin{itemize}
		\item[(i)] the functions $\nu, \mu: (0,\infty) \to [0,\infty)$ are continuous,
		\item[(ii)] $p \in C^2(0,\infty)$.
		%\item[(ii)] $p(0) = 0$ and $p(\infty) = \infty$.
		\end{itemize}}

\begin{theorem}\label{cNS-torus} Let $\Omega = \T^d$ and $(\varrho,\uu)$ be a weak solution of \eqref{gNS} with initial data \eqref{initial}. Assume that 
		\begin{equation} \label{Condition}
		0<c_1\leq \varrho(x,t)\leq c_2<\infty \quad \text{and} \quad \uu\in L^\infty(0,T;L^2(\mathbb{T}^d))\cap L^2(0,T;H^1(\mathbb{T}^d)).
		\end{equation}

	We assume in the case $d = 2$ that
\begin{equation} \label{rhod=2}
	\sup_{t\in (0,T)}\sup_{|h|<\varepsilon}  |h|^{-\frac 12}\|\varrho(\cdot+h,t)-\varrho(\cdot,t)\|_{L^{2}(\mathbb{T}^2)} <\infty,
	\end{equation}
	and in the case $d = 3$ that
	\begin{equation} \label{rhod=3}
	\begin{gathered}
	\uu\in  L^4(\T^3 \times (0,T)) \quad \text{ and } \quad	\sup_{t\in (0,T)}\sup_{|h|<\varepsilon} |h|^{-\frac 12}\|\varrho(\cdot+h,t)-\varrho(\cdot,t)\|_{L^{\frac{12}{5}}(\T^3)} <\infty.
	\end{gathered}
	\end{equation}
	
	 Then the energy equality holds, i.e.\blue{
	\begin{equation} \label{conser1}
	\begin{aligned}
	\int_{\mathbb{T}^d}\left(\frac{1}{2}(\varrho|\uu|^2)(x,t)+\blue{\HH(\vr)(x,t)}\right)dx&+\int_{0}^{t}\int_{\mathbb{T}^d}\nu(\vr)|\mathbb{D} \uu|^2dxds + \int_0^t\int_{\T^d}\mu(\vr)|\di \uu|^2dxds\\
	&= \int_{\mathbb{T}^d}\left(\frac{1}{2}\varrho_0|\uu_0|^2+\blue{\HH(\vr_0)}\right)dx \quad \forall t\in (0,T),
	\end{aligned}
	\end{equation}}
	\blue{where $\HH(\vr)$ is defined by
	\begin{equation}\label{GH}
		\HH(\vr) = \vr \int_1^{\vr}\frac{p(z)}{z^2}dz.
	\end{equation}}
	
\end{theorem} 

%{\color{red} (Comment about the additional term depending linearly on $\varrho$ in $\HH$ in \eqref{conser1}.)}

\begin{remark} %{\color{red} It should be modified. Discussion about the case $\nu$, $\mu$ are constants should be added.}
\blue{The conditions \eqref{rhod=2} and \eqref{rhod=3} can be obtained when $\vr$ belongs to the classical Besov spaces, more precisely $\vr \in L^\infty(0,T;B_2^{\frac{1}{2},\infty}(\T^2))$ when $d = 2$ and $\vr \in L^\infty(0,T;B_{\frac{12}{5}}^{\frac{1}{2},\infty}(\T^3))$ when $d = 3$, where $B_s^{r,\infty}(\T^d)$ is the classical Besov space for $s,r\in (0,\infty)$.}

\blue{It's worth to remark that the condition on $\vr$ depends on the dimension, but not on particular form of the pressure law.}
%\indent In general, the conditions \eqref{rhod=2} and \eqref{rhod=3} are satisfied if, for instance, $\vr\in L^\infty(0,T;B_{2}^{\frac{1}{2},\infty}(\T^2)\cap B_{p}^{1,\infty}(\T^2))$ and $\vr \in L^\infty(0,T;B_{\frac{12}{5}}^{\frac 12,\infty}(\T^3)\cap B_{p^{\frac 65}}^{1,\infty}(\T^3))$ respectively, where $B_\beta^{\alpha,\infty}(\T^d)$ is the classical Besov space for a constant $\beta > 0$ and $B_{p}^{\alpha,\infty}(\T^d)$ is the Besov-Orlicz space for a continuous, non-decreasing and convex function $p$ satisfying $p(0) = 0$ and $p(\infty) = \infty$.}
\end{remark}

\blue{
\begin{remark}
To prove the energy equality \eqref{conser1}, we only need the continuity of viscosity functions $\nu$ and $\mu$, since we only use the fact that $\nu(\vr)$ and $\mu(\vr)$ are bounded from above once $0< c_1 \leq \vr \leq c_2<\infty$. However, for the existence of solutions, in particular, to achieve the $L^2(0,T;H^1(\T^d))$ of the velocity, one might need to impose more conditions on either $\nu$ or $\mu$ to obtain such the regularity for $\uu$ from \eqref{conser1}. For instance, one can assume that $\nu$ is strictly increasing and $\nu(0) = 0$, which implies that $\nu(\vr) \geq \alpha_0 > 0$ once $\vr \geq c_1 > 0$.
\end{remark}}
\begin{remark}\label{class}
	If we have 
	\begin{equation}\label{classical}
	\uu \in L^s(0,T;L^q(\T^d)) \quad \text{ with } \quad \frac 1s + \frac 1q \leq \frac 12 \; \text{ and } \; q \geq 4
	\end{equation}
	 then by interpolation it follows
	\begin{equation*}
		\|\uu\|_{L^4(\T^d\times(0,T))} \leq C\|\uu\|_{L^\infty(0,T;L^2(\T^d))}^a\|\uu\|_{L^s(0,T;L^q(\T^d))}^{1-a}
	\end{equation*}
	for some $a\in (0,1)$. Therefore, the results of Theorem \ref{cNS-torus} (and the subsequent theorems) are also valid with the classical assumption \eqref{classical} on the velocity.
\end{remark}
Our results in Theorem \ref{cNS-torus} improve existing works in the literature, particularly in the following sense:
	\begin{itemize}
		\item \blue{We consider more general pressure laws as well as degenerate viscosities. This is possible thanks to the fact that vacuum is not presented.}
		\item In three dimensions, we require only $\uu \in L^4(\T^3\times(0,T))$, which is consistent with the classical conditions for homogeneous incompressible Navier-Stokes equations in e.g. \cite{Shi74} (See Remark \ref{class}). This improves for instance the results of  \cite{Yu17}, in which it was required that $\uu \in L^s(0,T;L^q(\T^3))$ with $1/s + 1/q \leq 5/12$ and $q\geq 6$, and only $\gamma$-pressure law, i.e. $p(\vr) = \vr^\gamma, \gamma>1$ and constant viscosities were considered.
		
		\medskip
		\item
		At first sight, conditions \eqref{rhod=2} and \eqref{rhod=3} seem stronger than what assumed in \cite[Theorem 1.1]{Yu17}. However, since \cite{Yu17} used the notion of weak solutions constructed in \cite{VY16} and \cite{LV18} which requires several regularity properties on the density $\vr$ and the velocity $\uu$. For instance, in \cite{Yu17}, the condition $\na\sqrt{\vr} \in L^\infty(0,T;L^2(\T^d))$ is imposed. This, in combination with $0< \underline \varrho \leq\vr(x,t) \leq \overline \varrho < \infty$, implies  that $\vr \in L^\infty(0,T;H^1(\T^d))$ which is clearly stronger than \eqref{rhod=2} or \eqref{rhod=3}.
	\end{itemize}
 
\medskip

Next we deal with the case where $\Omega$ is a bounded domain with $C^2$ boundary. 
For $\delta > 0$,  put
\begin{equation*}
\Omega_\delta: = \{x\in\Omega: d(x,\partial\Omega) > \delta \}.
\end{equation*} 
Since  $\Omega$  is a bounded, connected domain with $C^2$ boundary,  we find $r_0>0$ and a unique $C_b^1$-vector function $n:\Omega\backslash\Omega_{r_0}\to S^{d-1}$  such that the following holds true: for any $r\in [0,r_0)$, $x\in \Omega_{r}\backslash \Omega_{r_0}$ there exists  a unique $x_r\in\partial\Omega_{r}$  such that  $d(x,\partial \Omega_{r})=|x-x_r|$ and $n(x)$ is the outward  unit normal vector field to the boundary $\partial\Omega_r$ at $x_r$.  

The energy equality for \eqref{NS} in a domain with boundary is stated in the following theorem.

\begin{theorem}\label{cNS-domain}
	Let $\Omega$ be a bounded domain with $C^2$ boundary $\pa\Omega$ and $(\varrho,\uu)$ be a weak solution of \eqref{gNS} with initial data \eqref{initial} and Dirichlet boundary condition \eqref{boundary}. Assume that
	\begin{equation*}
		0<c_1\leq \vr(x,t) \leq c_2<\infty, \quad \text{ and } \quad \uu \in L^\infty(0,T;L^2(\Omega)) \cap L^2(0,T;H^1(\Omega)).
	\end{equation*}
	We assume in the case $d = 2$ that, for each $0 < \delta < 1$,
	\begin{equation}\label{domain-d2}
		\sup_{t\in(0,T)}\sup_{|h| \leq \delta} |h|^{-\frac 12}\|\varrho(\cdot+h,t)-\varrho(\cdot,t)\|_{L^{2}(\Omega_\delta)} <\infty, 
	\end{equation}
	and in case $d = 3$ that, for each $0 < \delta < 1$,
	\begin{equation}\label{domain-d3}
	\begin{gathered}
	\uu \in L^4(\Omega \times (0,T)) \quad \text{ and } \quad 
	\sup_{t\in(0,T)}\sup_{|h| \leq \delta} |h|^{-\frac 12}\|\varrho(\cdot+h,t)-\varrho(\cdot,t)\|_{L^{\frac{12}{5}}(\Omega_\delta)} <\infty.
	\end{gathered}
		\end{equation}
	Then the energy equality holds, i.e.\blue{
	\begin{equation}\label{eq_cNSd}
		\begin{aligned}
	\int_{\Omega}\left(\frac{1}{2}(\varrho|\uu|^2)(x,t)+\blue{\HH(\vr)(x,t)}\right)dx&+\int_{0}^{t}\int_{\Omega}\nu(\vr)\mathbb{D} \uu|^2dxds + \int_0^t\int_{\Omega}\mu(\vr)|\di \uu|^2dxds\\
	&= \int_{\Omega}\left(\frac{1}{2}\varrho_0|\uu_0|^2+\blue{\HH(\vr_0)}\right)dx \quad \forall t\in (0,T). 
		\end{aligned}                                     
	\end{equation}}
\end{theorem}

\begin{remark} In contrast to Euler equations, we do not need any additional regularities near the boundary to establish energy equality for Navier-Stokes equations in bounded domains. This is due to a careful use of homogeneous Dirichlet boundary conditions and the assumption $\uu \in L^2(0,T;H^1(\Omega))$ (see Lemma \ref{Poincare}).
%We note that the condition $\uu \in L^2(0,T,H^1(\Omega))$ is crucial to deal with the boundary effects, and hence in the statement of Theorem \ref{cNS-domain} no additional regularity hypothesis on $\uu$ near the boundary is needed. 
\end{remark}

As mentioned above, our method of proof is also suitable to obtain energy equalities for the inhomogeneous incompressible Navier-Stokes equation \eqref{icNS}.
%\begin{theorem}\label{cNS-normal-both}
%	Let either $\Omega = \T^d$ or $\Omega \subset \R^d$ be a bounded domain with $C^2$ boundary $\partial\Omega$ where $d=2,3$. Let $(\vr,\uu)$ be a weak solution to \eqref{cNS-normal} with initial data \eqref{initial} (and with Dirichlet boundary condition \eqref{boundary} in case $\Omega$ is a bounded domain). Assume  
%	\begin{equation*}
%		0 < c_1 \leq \vr(x,t) \leq c_2 < +\infty \quad \text{ and } \quad \uu \in L^\infty(0,T;L^2(\Omega))\cap L^2(0,T;H^1(\Omega)).
%	\end{equation*}
%	Moreover, if $\Omega=\T^d$ we assume \eqref{rhod=2} for $d=2$ and \eqref{rhod=3} for $d=3$, and if $\Omega$ is a bounded domain we assume \eqref{domain-d2} for $d=2$ and \eqref{domain-d3} for $d=3$. 
%	
%	Then the energy equality holds, i.e.
%	\begin{equation*}
%		\begin{aligned}
%	\int_{\Omega}\left(\frac{1}{2}(\varrho|\uu|^2)(x,t)+\blue{\HH(\vr)(x,t)}\right)dx&+2\nu\int_{0}^{t}\int_{\Omega}|\na \uu|^2dxdt + \lambda\int_0^t\int_{\Omega}|\di \uu|^2dxds\\
%	&= \int_{\Omega}\left(\frac{1}{2}\varrho_0|\uu_0|^2+\blue{\HH(\vr_0)}\right)dx \quad \forall t\in (0,T). 
%		\end{aligned}                                     
%	\end{equation*}	
%\end{theorem}
\begin{theorem}\label{icNS-both}
	Let either $\Omega = \T^d$ or $\Omega \subset \R^d$ be a bounded domain with $C^2$ boundary $\partial\Omega$ where $d=2,3$. Let $(\vr,\uu,P)$ be a weak solution to \eqref{icNS} with initial data \eqref{initial} (and with Dirichlet boundary condition \eqref{boundary} in case $\Omega$ is a bounded domain). Assume
	\begin{equation*}
		0 < c_1 \leq \vr(x,t) \leq c_2 < +\infty, \quad \uu \in L^\infty(0,T;L^2(\Omega))\cap L^2(0,T;H^1(\Omega)), \quad\text{and} \quad P\in L^2(\Omega \times (0,T)).
	\end{equation*}
%	In the case of bounded domain, we assume additionally
%	\begin{equation}\label{layer1}
%		\limsup_{\delta\to 0}\int_0^T\fint_{\Omega\backslash\Omega_\delta}\left(|\na \uu|^{\frac 32} + |P|^{\frac 32}\right)dxds < +\infty.
%	\end{equation}
	Moreover, in the case $d = 3$ we assume additionally $\uu \in L^4(\Omega\times (0,T))$. Then the energy equality holds, i.e.
	\blue{\begin{equation*}
		\frac 12 \intO \vr(x,t)|\uu(x,t)|^2dx + \int_0^t\intO\nu(\vr)|\mathbb D\uu|^2dxds = \frac 12\intO \vr_0(x)|\uu_0(x)|^2dx \quad \forall\, t\in (0,T).
	\end{equation*}}
\end{theorem}

\medskip
We emphasize again that, in Theorem \ref{icNS-both}, the condition $\uu \in L^2(0,T;H^1(\Omega))$ helps to handle the boundary effects without requiring extra conditions of $\uu$ near the boundary. As a consequence, we improve the results of \cite{Yu17-2} for homogeneous incompressible Navier-Stokes equations by removing the assumption $\uu \in L^s(0,T; B_s^{\alpha,\infty}(\Omega))$ with $\frac 12 + \frac 1s < \alpha < 1$.
\begin{theorem}\label{thm:org_NS}
	Let $\Omega \subset \mathbb R^3$ be a bounded domain with $C^2$ boundary and $(\uu, P)$ be a weak solution to the homogeneous incompressible Navier-Stokes equation
	\begin{equation}\label{org_NS}
	\begin{cases}
		\pa_t \uu - \mu \Delta \uu+ \di(\uu\otimes\uu) + \na P = 0, &\text{ in } \Omega\times(0,T),\\
		\di \uu = 0, &\text{ in } \Omega\times (0,T),\\
		\uu = 0, &\text{ on } \pa\Omega\times (0,T),\\
		\uu(x,0) = \uu_0(x), &\text{ in } \Omega,
	\end{cases}
	\end{equation}
	where $\mu > 0$ is the viscosity. Assume $\uu \in L^4(\Omega\times(0,T))$ and there exists $\delta_0>0$ such that 
	\begin{equation}\label{pressure}
		P\in L^2(\Omega\backslash\Omega_{\delta_0}\times(0,T)).
	\end{equation}

Then the energy equality holds
	\begin{equation*}
		\int_{\Omega}|\uu(x,t)|^2dx + \mu\int_0^t\int_{\Omega}|\na \uu(x,s)|^2dxds = \int_{\Omega}|\uu_0(x)|^2dx \quad \forall\, t\in (0,T).
	\end{equation*}
\end{theorem}
\begin{remark}
	The assumption \eqref{pressure} is to deal with the boundary layer. We remark that \cite{Yu17-2} did not impose any condition on the pressure, but the author used $P = 0$ on the boundary (see \cite[Proposition 2.3]{Yu17-2}) which is neither assumed nor implied from \eqref{org_NS}. Nevertheless, if somehow $P = 0$ on the boundary (in a weak sense), from the equation $-\Delta P = \sum_{i,j=1}^{d}\pa_{x_i}\pa_{x_j}(\uu_i\uu_j)$ we obtain $\|P\|_{L^2(\Omega\times(0,T))} \leq C\|\uu\|_{L^4(\Omega\times(0,T))}^2$ and therefore \eqref{pressure} is automatically satisfied.
\end{remark}

\begin{remark}
	After the completion of this work, we came to know the recent preprint of Chen, Liang, Wang and Xu \cite{CLWX} where the authors also use the concept of weak solutions defined in Definition \ref{solution}. By using a different approach, they obtain therein the energy equality for \eqref{cNS-normal} under the conditions, among others, on velocity $\uu \in L^s(0,T;L^q(\Omega))$ for $s \geq 4$ and $q\geq 6$, and on density $\nabla \sqrt\varrho \in L^{\infty}(0,T;L^2(\Omega))$ which are stronger than our assumptions in Theorem \ref{cNS-torus}. Nevertheless, they are able to deal with the case of vacuum, i.e. $\varrho$ is only assumed to be nonnegative instead of bounded below by a strict positive constant.
\end{remark}
\medskip
{\bf The organization of this paper} is as follows: In the next section, we prove some auxiliary estimates which will play important roles in our proof. The proofs of Theorems \ref{cNS-torus}, \ref{cNS-domain}, \ref{icNS-both}, and \ref{thm:org_NS} are presented in the last four sections respectively. \medskip

{\bf Notation.} Throughout the paper, $C$ denotes generic constants which may depend on $d$, $T$, $\|\varrho\|_{L^\infty(\Omega \times (0,T))}, \|\varrho^{-1}\|_{L^\infty(\Omega \times (0,T))}$ and other scalar parameters. We use the notation $\|f(s)\|_{L^p(\Omega)}$ to denote $\| f(\cdot,s)\|_{L^p(\Omega)}$.

For any Borel set $E$, we denote by $\fint_E f(x)dx = \frac{1}{\mathcal L^d(E)}\int_{E}f(x)dx$ the average of $f$ over $E$, where $\mathcal L^d(E)$ is the Lebesgue measure of $E$.

\section{Preliminaries}

Let $\omega: \mathbb R^d \to \mathbb R$ be a standard mollifier, i.e. $\omega(x) = c_0e^{-\frac{1}{1-|x|^2}}$ for $|x| < 1$ and $\omega(x) = 0$ for $|x| \geq 1$, where $c_0$ is a constant such that $\int_{\R^d}\omega(x)dx = 1$. For any $\eps > 0$, we define the rescaled mollifier $\omega_\varepsilon(x) = \eps^{-d}\omega(\frac x\eps)$. For any function $f \in L_{loc}^1(\Omega)$, its mollified version is defined as
\begin{equation*}
	f^\eps(x) = (f\star \omega_\eps)(x) = \int_{\mathbb R^d}f(x-y)\omega_\eps(y)dy, \quad x \in \Omega_\varepsilon,
\end{equation*}
recalling
$\Omega_\varepsilon = \{x\in\Omega: d(x,\partial\Omega) > \varepsilon \}$. 
\begin{lemma} \label{lemma1} Let $2 \leq d \in {\mathbb N}$, $1 \leq p,q \leq \infty$ and $f: \mathbb{T}^d\times(0,T) \to {\mathbb R}$.  
\begin{itemize}
\item[(i)] Assume $f \in L^p(0,T;L^q(\T^d))$. Then for any $\varepsilon> 0$, there holds
\begin{equation} \label{o1-1}
\| f^\varepsilon  \|_{L^p(0,T;L^\infty(\T^d))} \leq C\varepsilon^{- \frac{d}{q}}\| f \|_{L^p(0,T;L^q(\T^d))},
\end{equation}
\begin{equation} \label{o1-2}
\|\nabla  f^\varepsilon  \|_{L^p(0,T;L^\infty(\T^d))} \leq C\varepsilon^{-1- \frac{d}{q}}\| f \|_{L^p(0,T;L^q(\T^d))}.
\end{equation}

\item[(ii)] Assume $f \in L^p(0,T;L^q(\T^d))$. Then for any $\varepsilon> 0$, there holds
	\begin{equation} \label{o2}
	\|\nabla f^\varepsilon\|_{L^p(0,T;L^q(\mathbb{T}^d))}\leq C \varepsilon^{-1}  \|f\|_{L^p(0,T;L^q(\mathbb{T}^d))}.
	\end{equation}
Moreover, if $p,q<\infty$ then
\begin{equation} \label{lim2}
\limsup_{\varepsilon \to 0} \varepsilon\| \nabla f^\varepsilon \|_{L^p(0,T;L^q(\T^d))} = 0.
\end{equation}	
\item[(iii)]	 Assume $f \in L^p(0,T;L^q(\mathbb{T}^d))$ and $g: \T^d\times(0,T)\to\R$ with $0<c_1\leq g\leq c_2<\infty$. Then for any $\varepsilon> 0$, there holds
	\begin{equation}\label{z1}
	\left\|\nabla \frac{f^\varepsilon}{g^\varepsilon}\right\|_{L^p(0,T;L^q(\mathbb{T}^d))}\leq C(c_1,c_2) \varepsilon^{-1}  \|f\|_{L^p(0,T;L^q(\Omega))}.
	\end{equation}
	Moreover, if $p,q<\infty$ then
	\begin{equation}\label{z1'}
\limsup_{\eps \to 0}\varepsilon	\left\|\nabla \frac{f^\varepsilon}{g^\varepsilon}\right\|_{L^p(0,T;L^q(\mathbb{T}^d))}=0.
	\end{equation}
\item[(iv)] Assume $f \in L^2(0,T;H^1(\T^2))$. Then for any $\varepsilon> 0$, there holds
\begin{equation} \label{o3}
\| \nabla f^\varepsilon \|_{L^2(0,T;L^\infty(\T^2))} \leq C \varepsilon^{-1}\| f \|_{L^2(0,T;H^1(\T^2))}.
\end{equation}
Moreover, for any $r\in [1,2]$, there holds
\begin{equation} \label{lim3} 
\limsup_{\varepsilon \to 0} \varepsilon \| \nabla f^\varepsilon \|_{L^r(0,T;L^\infty(\T^2))} = 0.
\end{equation}
\end{itemize}

\end{lemma}
\begin{proof}

(i) By the definition of $f^\varepsilon$ and H\"older's inequality, for a.e. $x \in \T^d$ and $s \in (0,T)$, we have
\begin{align*}
	|f^\eps(x,s)| \leq \intTd|f(x-y,s)\omega_{\eps}(y,s)|dy &\leq \left(\intTd|f(x-y,s)|^qdy \right)^{\frac{1}{q}}\left(\intTd|\omega_\eps(y)|^{\frac{q}{q-1}}dy \right)^{\frac{q-1}{q}}\\
	&\leq C\eps^{-\frac dq}\|f(s)\|_{L^q(\T^d)},
\end{align*}
where we have used $\omega_\eps(x) = \eps^{-d}\omega(x/\eps)$ at the last step. This implies 
$$ \| f^\varepsilon(s) \|_{L^\infty(\T^d)} \leq C \varepsilon^{-\frac{d}{q}}\| f(s) \|_{L^q(\T^d)},
$$
which in turn yields \eqref{o1-1}. 

Next we use H\"older's inequality again to estimate
\begin{equation*}
	|\nabla f^\eps(x,s)| \leq \intTd|f(x-y,s)||\na \omega_\eps(y)|dy \leq C\eps^{-1-\frac dq}\|f(s)\|_{L^q(\T^d)}.
\end{equation*}
This leads to \eqref{o1-2}. 

\medskip
(ii) From the fact that $\int_{\mathbb{R}^d}\nabla \ox_\varepsilon(y)dy=0$ and H\"older's inequality, we obtain, for a.e. $x \in \T^d$ and $s \in (0,T)$,
\begin{align*}  
|\nabla f^\varepsilon(x,s)| &\leq \left|\int_{|y|<\varepsilon}[f(x-y,s)-f(x,s)]\nabla \omega_\varepsilon(y)dy \right| \\
%&\leq \left(\int_{|y|<\varepsilon}\left|g(x-y)-g(x)\right|^{p}dy\right)^{\frac{1}{p}} \left( \int_{\mathbb{R}^d} | \nabla \ox_\varepsilon(y)|^{p'}dy\right)^{\frac{1}{p'}} \\
&\leq C\varepsilon^{\frac{d}{q}}\left(\fint_{|y|<\varepsilon}|f(x-y,s)-f(x,s)|^{q}dy\right)^{\frac{1}{q}} \varepsilon^{-1-\frac{d}{q}}\|\na \omega\|_{L^{\frac{q}{q-1}}(\T^d)}.
\end{align*}
It follows that
\begin{equation} \label{h1}
\| \nabla f^\varepsilon(s) \|_{L^q(\T^d)} \leq C\varepsilon^{-1}\left( \fint_{|y|<\varepsilon} \int_{\T^d}\left|f(x-y,s)-f(x,s)\right|^{q}dxdy \right)^{\frac{1}{q}}. 
 \end{equation}
This implies \eqref{o2}. 
Thus, for any $g\in C^\infty(\T^d\times (0,T))$, we have 
\begin{equation*} 
\varepsilon\| \nabla f^\varepsilon \|_{L^p(0,T;L^q(\T^d))} \leq C\varepsilon\|g\|_{C^1_b(\T^d\times(0,T))}+C\|f-g\|_{L^p(0,T;L^q(\mathbb{T}^d))}
\end{equation*}
Therefore, for any $g\in C^\infty(\T^d\times(0,T))$,
\begin{align*}
\limsup_{\varepsilon \to 0} \varepsilon\| \nabla f^\varepsilon \|_{L^p(0,T;L^q(\T^d))} \leq C\|f-g\|_{L^p(0,T;L^q(\mathbb{T}^d))}
\end{align*}
and since $C^\infty(\T^d\times(0,T))$ is dense in $L^p(0,T;L^q(\T^d))$ (for $p, q< \infty$), we obtain \eqref{lim2}.
%
%Put $h_\varepsilon(s):=\varepsilon \| \nabla f^\varepsilon(s) \|_{L^q(\T^d)}$ for $s \in (0,T)$. {\color{red} Assume $p,q<\infty$}, by using density argument and \eqref{h1}, we deduce that $h_\varepsilon(s) \to 0$ as $\varepsilon \to 0$ for a.e. $s \in (0,T)$. On the other hand, we infer from \eqref{o2} that $\{h_\varepsilon\}$ is uniformly bounded in $L^p((0,T))$. This, together with the fact that $r<p$ and H\"older's inequality, implies that $\{h_\varepsilon^r\}$ is equi-integrable. Therefore, by invoking Vitali convergence theorem, we deduce that $h_\varepsilon \to 0$ in $L^r(0,T))$, which is yields \eqref{lim2}. 

\medskip
(iii) We have, for any $g_1,f_1\in C^\infty_b(\T^d\times(0,T))$, 
\begin{align*}
	\left\|\nabla \frac{f^\varepsilon}{g^\varepsilon}\right\|_{L^p(0,T;L^q(\mathbb{T}^d))}&\leq C\varepsilon^{-1}  \|f-f_1\|_{L^p(0,T;L^q(\Omega))}+C\|\nabla g^\varepsilon\|_{L^p(0,T;L^q(\mathbb{T}^d))}\|f_1\|_{L^\infty(\T^d\times(0,T))}\\&
	\leq C\varepsilon^{-1}  \|f-f_1\|_{L^p(0,T;L^q(\Omega))}+C\varepsilon^{-1}\| g-g_1\|_{L^p(0,T;L^q(\mathbb{T}^d))}\|f_1\|_{L^\infty(\T^d\times(0,T))}\\
	&\quad +C\| g_1\|_{C^1_b(\T^d\times(0,T))}\|f_1\|_{L^\infty(\T^d\times(0,T))}
\end{align*}

Thus, by density this implies \eqref{z1} and \eqref{z1'}.

\medskip
(iv) From the fact that $\int_{\mathbb{R}^2}\nabla \ox_\varepsilon(y)dy=0$ and H\"older's inequality, we obtain, for a.e. $x \in \T^d$ and $s \in (0,T)$,  
\begin{equation} \label{h2}\begin{aligned} 
|\nabla f^\varepsilon(x,s)| &\leq \left(\int_{|y|<\varepsilon}\left|f(x-y,s)-f(x,s)\right|^{2}dy\right)^{\frac{1}{2}} \left( \int_{\mathbb{R}^2} | \nabla \ox_\varepsilon(y)|^{2}dy\right)^{\frac{1}{2}}\\ 
&\leq C \varepsilon^{-1} \left( \int_0^1 \int_{|y| <\varepsilon} |\nabla f(x + \rho y,s)|^2 dy d\rho  \right)^{\frac{1}{2}}  \|\nabla\ox\|_{L^2(\mathbb{R}^2)}. 
\end{aligned} \end{equation}
This implies \eqref{o3}. Thus, as (ii), we have, for any  $g\in C^\infty(\T^d\times(0,T))$, 
\begin{equation} 
\limsup_{\varepsilon \to 0} \varepsilon \| \nabla f^\varepsilon \|_{L^2(0,T;L^\infty(\T^2))} \leq C \| f-g \|_{L^2(0,T;H^1(\T^2))}.
\end{equation} 
By using the fact that $C^\infty(\T^d\times(0,T))$ is dense in $L^2(0,T;H^1(\T^d))$, 
we derive 
$$ \limsup_{\varepsilon \to 0} \varepsilon \| \nabla f^\varepsilon \|_{L^2(0,T;L^\infty(\T^2))}=0, $$
which implies \eqref{lim3}. The proof is complete.
\end{proof}

The following variant of Lemma \ref{lemma1} in bounded domains can be proved similarly, so we only state the results. We recall that, for each $\delta > 0$,  $\Omega_\delta = \{x\in\Omega: d(x,\partial\Omega) > \delta \}$.

\begin{lemma} \label{lemma2}
Let $2 \leq d \in {\mathbb N}$,  $\Omega\subset\mathbb R^d$ be a bounded domain with $C^2$ boundary $\partial\Omega$, $1 \leq p,q \leq \infty$ and $f: \Omega\times(0,T) \to {\mathbb R}$.  
\begin{itemize}
\item[(i)] Assume $f \in L^p(0,T;L^q(\Omega))$. Then for any $0<\varepsilon <\delta$, there holds
\begin{equation} \label{o1-1b}
\| f^\varepsilon  \|_{L^p(0,T;L^\infty(\Omega_\delta))} \leq C\varepsilon^{- \frac{d}{q}}\| f \|_{L^p(0,T;L^q(\Omega))},
\end{equation}
\begin{equation} \label{o1-2b}
\|\nabla  f^\varepsilon  \|_{L^p(0,T;L^\infty(\Omega_\delta))} \leq C\varepsilon^{-1- \frac{d}{q}}\| f \|_{L^p(0,T;L^q(\Omega))}.
\end{equation}

\item[(ii)] Assume $f \in L^p(0,T;L^q(\Omega))$. Then for any $0<\varepsilon <\delta$, there holds
\begin{equation} \label{o2b}
\|\nabla f^\varepsilon\|_{L^p(0,T;L^q(\Omega_\delta))}\leq C \varepsilon^{-1}  \|f\|_{L^p(0,T;L^q(\Omega))}.
\end{equation}
Moreover, if $p, q < \infty$ then
\begin{equation} \label{lim2b}
\limsup_{\varepsilon \to 0} \varepsilon\| \nabla f^\varepsilon \|_{L^p(0,T;L^q(\Omega_\delta))} = 0.
\end{equation}	
\item[(iii)] Assume $f \in L^p(0,T;L^q(\Omega))$ with $p,q<\infty$ and $g: \Omega\times(0,T)\to \R$ with $0<c_1\leq g\leq c_2<\infty$. Then for any $0<\varepsilon <\delta$, there holds
\begin{equation}
\left\|\nabla \frac{f^\varepsilon}{g^\varepsilon}\right\|_{L^p(0,T;L^q(\Omega_\delta))}\leq C(c_1,c_2) \varepsilon^{-1}  \|f\|_{L^p(0,T;L^q(\Omega))}.
\end{equation}
\item[(iv)] Assume $d=2$ and $f \in L^2(0,T;H^1(\Omega))$. Then for any $0<\varepsilon <\delta$, there holds
\begin{equation} \label{o3b}
\| \nabla f^\varepsilon \|_{L^2(0,T;L^\infty(\Omega_\delta))} \leq C \varepsilon^{-1}\| f \|_{L^2(0,T;H^1(\Omega))}.
\end{equation}
Moreover, for any $r \in [1,2]$, there holds
\begin{equation} \label{lim3b} 
\limsup_{\varepsilon \to 0} \varepsilon \| \nabla f^\varepsilon \|_{L^r(0,T;L^\infty(\Omega_\delta))} = 0.
\end{equation} 
\end{itemize}
\end{lemma}

\begin{lemma}\footnote{Similar estimates were obtained in \cite{CET94} in the context of H\"older spaces.}\label{keyla} Let $p,p_1\in [1,\infty)$ and $p_2\in (1,\infty]$ with $\frac{1}{p}=\frac{1}{p_1}+\frac{1}{p_2}$. Assume $f\in L^{p_1}(0,T;W^{1,p_1}(\mathbb{T}^d))$ and $g\in L^{p_2}(\T^d \times (0,T))$. Then for any $\varepsilon> 0$, there holds
	\begin{align} \label{fg'}
	\|(fg)^\varepsilon-f^\varepsilon g^\varepsilon\|_{L^p(\T^d \times (0,T))}\leq C\varepsilon \|f\|_{L^{p_1}(0,T;W^{1,p_1}(\T^d))}\|g\|_{L^{p_2}(\T^d \times (0,T))}.
	\end{align}
	Moreover, if $p_2<\infty$ then 
	\begin{align}\label{limite'}
	\limsup_{\varepsilon \to 0}\varepsilon^{-1} \|(fg)^\varepsilon-f^\varepsilon g^\varepsilon\|_{L^p(\T^d \times (0,T))}=0.
	\end{align}
\end{lemma}
\begin{proof}
 We note that (see e.g. \cite{CET94})
	\begin{align}\label{z2}
	(fg)^\varepsilon-f^\varepsilon g^\varepsilon=R^\varepsilon-(f^\varepsilon-f)(g^\varepsilon-g)
	\end{align}
	where
	\begin{align*}
	R^\varepsilon(x,s):=\int_{\mathbb{R}^d}(f(x-y,s)-f(x,s)) (g(x-y,s)-g(x,s))\ox_\varepsilon(y)dy.
	\end{align*} This yields
	\begin{align}\label{z3}
	\|(fg)^\varepsilon-f^\varepsilon g^\varepsilon\|_{L^p(\T^d \times (0,T))}\leq \|R^\varepsilon\|_{L^p(\T^d \times (0,T))}+\|(f-f^\varepsilon)(g-g^\varepsilon)\|_{L^p(\T^d \times (0,T))}.
	\end{align}	
The first term on the right hand-side of \eqref{z3} can be estimated using H\"older's inequality as 
 \begin{align} \label{lima} \nonumber
	&\|R^\varepsilon\|_{L^p(\T^d \times (0,T))}\\
	&\leq C\varepsilon \left[\int_{0}^{T}\!\!\!\int_{0}^{1}\!\!\!\fint\limits_{|y|<\varepsilon}\!\!\int_{\mathbb{T}^d} |\nabla f (x+\rho y,s)|^{p_1} dxdyd\rho ds \right]^{\frac{1}{p_1}} \left[ \int_{0}^{T}\!\!\!\!\!\!\fint\limits_{|y|<\varepsilon}\!\!\int_{\mathbb{T}^d}  |g(x-y,s)-g(x,s)|^{p_2}  dxdy ds \right]^{\frac{1}{p_2}}.
	\end{align}	 
Similarly, one can show that the second term on the right hand-side of \eqref{z3} is bounded from above by the term on the right hand-side of \eqref{lima}.
% \begin{align*}
%	&\|(f-f^\varepsilon)(g-g^\varepsilon)\|_{L^p(\T^d)} \\
%&\leq C\varepsilon^{1-\frac{d}{p}} \left[\int_{0}^{T}\!\!\!\int_{0}^{1}\!\!\!\int\limits_{|y|<\varepsilon}\!\!\int_{\mathbb{T}^d} |\nabla f (s,x+\rho y)|^{p_1} dxdyd\rho ds \right]^{\frac{1}{p_1}} \left[ \int_{0}^{T}\!\!\!\!\!\!\int\limits_{|y|<\varepsilon}\!\!\int_{\mathbb{T}^d}  |g(s,x-y)-g(s,x)|^{p_2}  dxdy ds \right]^{\frac{1}{p_2}}.
%\end{align*}	 
Therefore
	\begin{align} \label{lep} \nonumber
	&\|(fg)^\varepsilon-f^\varepsilon g^\varepsilon\|_{L^p(\T^d \times (0,T))} \\ 
	&\leq C\varepsilon \left[\int_{0}^{T}\!\!\!\int_{0}^{1}\!\!\!\fint\limits_{|y|<\varepsilon}\!\!\int_{\mathbb{T}^d} |\nabla f (x+\rho y,s)|^{p_1} dxdyd\rho ds \right]^{\frac{1}{p_1}} \left[ \int_{0}^{T}\!\!\!\!\!\!\fint\limits_{|y|<\varepsilon}\!\!\int_{\mathbb{T}^d}  |g(x-y,s)-g(x,s)|^{p_2}  dxdy ds \right]^{\frac{1}{p_2}}.
	\end{align}	
	This implies \eqref{fg'}.  	
By using \eqref{lep} and the density argument, we get \eqref{limite'}.	
\end{proof}

Similar results can be obtained for bounded domains using only minor modifications.
\begin{lemma}\label{keylabd} Let $\Omega\subset\mathbb R^d$ be a bounded domain with $C^2$ boundary $\partial\Omega$, $p,p_1\in [1,\infty)$ and $p_2\in (1,\infty]$ with $\frac{1}{p}=\frac{1}{p_1}+\frac{1}{p_2}$. Assume $f\in L^{p_1}(0,T;W^{1,p_1}(\Omega))$ and $g\in L^{p_2}(\Omega \times (0,T))$. Then for any $0<\varepsilon<\delta$ small, there holds
	\begin{align} \label{fg}
	\|(fg)^\varepsilon-f^\varepsilon g^\varepsilon\|_{L^p( \Omega_{2\delta} \times (0,T))}\leq C\varepsilon \|f\|_{L^{p_1}(0,T;W^{1,p_1}(\Omega_\delta))}\|g\|_{L^{p_2}( \Omega_\delta \times (0,T))}.
	\end{align}
	Moreover, if $p_2<\infty$ then 
	\begin{align}\label{limite}
	\limsup_{\varepsilon \to 0}\varepsilon^{-1} \|(fg)^\varepsilon-f^\varepsilon g^\varepsilon\|_{L^p( \Omega_{2\delta} \times (0,T))}=0.
	\end{align}
\end{lemma}

The following interpolation lemma shows that the additional assumption $\uu \in L^4(\Omega \times (0,T))$ is only needed in three dimensions.
\begin{lemma} \label{GN}  
Let either $\Omega = \mathbb T^2$ or $\Omega\subset\mathbb R^2$ be a bounded domain with $C^2$ boundary. Assume $f \in L^\infty(0,T;L^2(\Omega)) \cap L^2(0,T;H^1(\Omega))$. Then $f \in L^4(\Omega \times (0,T))$ and
\begin{equation} \label{GN1}
\| f \|_{L^4(\Omega \times (0,T))} \leq C \| f\|_{L^\infty(0,T;L^2(\Omega))}^{\frac{1}{2}} \| f \|_{L^2(0,T;H^1(\Omega))}^{\frac{1}{2}}.
\end{equation}
\end{lemma}
\begin{proof} 
The proof follows directly from Gargliardo-Nirenberg's inequality so we omit it.
\end{proof}

To deal with domains with boundary, we need the following coarea formula for $0<\kappa_1<\kappa_2$
\begin{equation}\label{coarea}
\int_{\kappa_1}^{\kappa_2}\int_{\pa\Omega_{\kappa}}g(\theta)\dH(\theta)d\kappa = \int_{\Omega_{\kappa_1}\backslash\Omega_{\kappa_2}}g(x)dx.
\end{equation}
%\begin{lemma}
%	Let $\Omega\subset\mathbb R^d$ and $f\in L^2(0,T;H^1(\Omega))\cap L^4((0,T)\times\Omega)$ with $f = 0$ on $(0,T)\times\pa\Omega$ (in the trace sense). Then for any $1\leq q \leq 3$ we have
%	\begin{equation}\label{trace}
%		\lim_{\delta\to 0}\int_0^T\fint_{\Omega\backslash\Omega_\delta}|f|^qdxds = 0.
%	\end{equation}
%\end{lemma}
%\begin{proof}
%	It is sufficient to only prove for $q = 3$. Since $\partial\Omega$ is of $C^2$ class, {\color{blue} for small $\delta$ and each $x\in \Omega\backslash\Omega_\delta$ there exists a unique $x_\delta \in \pa\Omega$ such that $|x - x_\delta| = d(x,\pa\Omega) \leq \delta$.} By direct computations, using $f = 0$ on $\pa\Omega$ and H\"older's inequality
%	\begin{align*}
%		&\int_0^T\fint_{\Omega\backslash\Omega_\delta}|f(x,s)|^3dxds\\
%		&= 
%		\frac{1}{\delta}\int_0^T\int_{\Omega\backslash\Omega_\delta}|f(x,s)|^2|f(x,s) - f(x_\delta,s)|dxds\\
%		&\leq \frac{1}{\delta}\int_0^T\int_{\Omega\backslash\Omega_\delta}|f(x,s)|^2\int_0^1|\na f(x_\delta + r(x-x_\delta))| |(x - x_\delta)| drdxds\\
%		&\overset{|x - x_\delta|\leq \delta}{\leq} \left(\int_0^T\int_{\Omega\backslash\Omega_\delta}|f(x,s)|^4dxds \right)^{1/2}{\color{red}\left(\int_0^T\int_{\Omega\backslash\Omega_\delta}\int_0^1|\na f(x_\delta + r(x-x_\delta),s)|^2drdxds \right)^{1/2} }.
%	\end{align*}
%	This point is not clear!! Note that $x_\delta=K(x)$, you need to make a change of variables!!!.\\
%	Since $f\in L^2(0,T;H^1(\Omega))\cap L^4((0,T)\times\Omega)$, we see that the right hand side vanishes as $\delta \to 0$, which finishes the proof.
%\end{proof}

\begin{lemma} \label{Poincare} Let $\Omega$ be a bounded domain with $C^2$ boundary. Assume that $f \in L^2(0,T;H_0^1(\Omega))$. Then, for $\delta>0$ small,
\begin{align}\label{z200}
\| f\|_{L^2((\Omega\setminus \Omega_{\delta}) \times (0,T))} \leq C \delta \| \nabla f \|_{L^2((\Omega \setminus \Omega_{2\delta}) \times (0,T))}.
\end{align}	
\end{lemma}
\begin{proof} Let $\delta>0$ be small. For any $x\in \Omega\backslash\Omega_{\delta}$, there exists a unique $x_{\partial \Omega}$ such that $|x_{\partial \Omega}-x|=d(x,\partial \Omega)$. Let $\mathcal{T}$ be the projection mapping defined by $\mathcal{T}(x):=x_{\partial \Omega}$. Then we have 
\begin{align}\label{z100}
||\nabla \mathcal{T}-\mathbf{id}||_{L^\infty(\Omega\backslash\Omega_{\delta})}=\circ(1) ~\text{as}~\delta\to 0.
\end{align}
Since $f=0$ on $\partial \Omega \times (0,T)$, by the coarea formula, we have
\begin{equation} \label{poin1} \begin{aligned}
\int_{\Omega\backslash\Omega_{\delta}}| f(x,s)|^2dx&=\int_0^\delta \int_{\partial \Omega_\kappa} |f (\theta,s)|^2 d\mathcal{H}^{d-1}(\theta) d\kappa\\
&=\int_0^\delta \int_{\partial \Omega_\kappa} |f(\theta,s)-f (\mathcal{T}(\theta),s)|^2 d\mathcal{H}^{d-1}(\theta) d\kappa
\\&\leq \int_0^\delta \int_{\partial \Omega_\kappa}\int_{0}^1 |\nabla f (\theta+\rho (\mathcal{T}(\theta)-\theta),s)|^2|\mathcal{T}(\theta)-\theta|^2 d\rho d\mathcal{H}^{d-1}(\theta) d\kappa
\\&\leq \delta^2 \int_0^\delta \int_{\partial \Omega_\kappa}\int_{0}^1 |\nabla f (\theta+\rho (\mathcal{T}(\theta)-\theta),s)|^2 d\rho d\mathcal{H}^{d-1}(\theta) d\kappa
\\&= \delta^2 \int_{0}^1\int_{\Omega\backslash\Omega_{\delta}} |\nabla f (x+\rho (\mathcal{T}(x)-x),s)|^2 dx d\rho.
\end{aligned} \end{equation}
Set $\mathcal{T}_\rho(x)=x+\rho (\mathcal{T}(x)-x)$. From \eqref{z100}, we deduce, for $\delta>0$ small and $\rho \in (0,1)$, that
\begin{align*}
\frac{1}{2}\leq |\det (\nabla \mathcal{T}_\rho(x)) |\leq \frac{3}{2} \quad \text{and} \quad  \mathcal{T}_\rho(\Omega\backslash\Omega_{\delta})\subset \Omega\backslash\Omega_{2\delta}
\end{align*}
Therefore
\begin{equation} \label{poin2} \begin{aligned}
\int_{0}^1\int_{\Omega\backslash\Omega_{\delta}} |\nabla f (x+\rho (\mathcal{T}(x)-x),s)|^2 dx d\rho& = \int_{0}^1\int_{\mathcal{T}_\rho(\Omega\backslash\Omega_{\delta})} |\nabla f (x,s)|^2 \frac{dx}{|\det (\nabla \mathcal{T}_\rho)(\mathcal{T}_\rho^{-1}(x)) |} d\rho \\ 
&\leq 2\int_{\Omega\backslash\Omega_{2\delta}} |\nabla f (x,s)|^2 dx.
\end{aligned} \end{equation}
Therefore, from \eqref{poin1} and \eqref{poin2} we derive \eqref{z200}.	
\end{proof}

\section{Proof of Theorem \ref{cNS-torus}}
By smoothing \eqref{NS}, we obtain
	\begin{equation}\label{Trho_approx}
	\pa_t\vre + \di(\vr \uu)^\eps = 0
	\end{equation}
	and
	\begin{equation}\label{Tu_approx}
	\pa_t(\vr\uu)^\eps + \di(\vr\uu\otimes\uu)^\eps + \nabla (p(\vr))^\varepsilon -  \nabla \cdot (\nu(\varrho) \D \uu)^\varepsilon  - \nabla(\mu(\vr)\di \uu)^\eps= 0
	\end{equation}
	for any $0<\varepsilon<1$. 
	
Multiplying \eqref{Tu_approx} by $(\varrho^\varepsilon)^{-1}(\varrho \uu)^\varepsilon$ then integrating on $(\tau,t)\times\mathbb T^d$, for $0<\tau<t<T$, we get
\begin{equation} \label{multiply} \begin{aligned}
&\int_{\tau}^{t}\int_{\mathbb{T}^d}  \frac{(\varrho \uu)^\varepsilon}{\varrho^\varepsilon}\partial_t(\varrho \uu)^\varepsilon dxds + \int_{\tau}^{t}\int_{\mathbb{T}^d} \frac{(\varrho \uu)^\varepsilon}{\varrho^\varepsilon}\nabla \cdot (\varrho \uu \otimes\uu)^\varepsilon dxds + \blue{\int_{\tau}^{t}\int_{\mathbb{T}^d} \frac{(\varrho \uu)^\varepsilon}{\varrho^\varepsilon}\nabla (p(\vr))^\varepsilon dxds} \\
&\blue{-  \int_{\tau}^{t}\int_{\mathbb{T}^d} \nabla \cdot (\nu(\varrho) \mathbb{D}\uu)^\varepsilon \frac{(\varrho \uu)^\varepsilon}{\varrho^\varepsilon} dxds - \int_{\tau}^{t}\int_{\mathbb{T}^d} \nabla (\mu(\varrho) \di\uu)^\varepsilon \frac{(\varrho \uu)^\varepsilon}{\varrho^\varepsilon} dxds}= 0.
\end{aligned} \end{equation}
Denote by $(A), (B), (C), (D)$ and $(E)$ the terms on the left-hand side of \eqref{multiply} respectively. We will estimate them below. 
\subsection{Estimate of $(A)$} By using \eqref{Trho_approx}, we can compute
\begin{align*}
(A)&=\frac{1}{2} \int_{\tau}^{t}\int_{\mathbb{T}^d} \partial_t\left(\frac{|(\varrho \uu)^\varepsilon|^2}{\varrho^\varepsilon}\right)dxds + \frac{1}{2} \int_{\tau}^{t}\int_{\mathbb{T}^d} \partial_t \varrho^\varepsilon\frac{|(\varrho \uu)^\varepsilon|^2}{(\varrho^\varepsilon)^2}dxds \\
&= \frac{1}{2} \int_{\tau}^{t}\int_{\mathbb{T}^d} \partial_t\left(\frac{|(\varrho \uu)^\varepsilon|^2}{\varrho^\varepsilon}\right)dxds-\frac{1}{2} \int_{\tau}^{t}\int_{\mathbb{T}^d} \nabla \cdot (\varrho \uu)^\varepsilon\frac{|(\varrho \uu)^\varepsilon|^2}{(\varrho^\varepsilon)^2}dxds\\
&=: (A1) + (A2).
\end{align*}
We see that $(A1)$ is the desired term while $(A2)$ will be canceled with the term $(B3)$ later.
%It follows that
%\begin{align*}
%\limsup_{\varepsilon \to 0}\limsup_{\tau \to 0}\left|(A)- \frac{1}{2}\int_{\mathbb{T}^d} \frac{|(\varrho \uu)^\varepsilon|^2}{\varrho^\varepsilon} (x,t)dx+\frac{1}{2}\int_{\mathbb{T}^d} \frac{|(\varrho \uu)^\varepsilon|^2}{\varrho^\varepsilon} (x,\tau)dx+\frac{1}{2} \int_{\tau}^{t}\int_{\mathbb{T}^d} \nabla \cdot (\varrho \uu)^\varepsilon\frac{|(\varrho \uu)^\varepsilon|^2}{(\varrho^\varepsilon)^2}dxds\right| = 0.
%\end{align*}
%Thus
%\begin{equation} \label{eA1}
%\limsup_{\varepsilon \to 0}\left|(A)- \frac{1}{2}\int_{\mathbb{T}^d} (\varrho|\uu|^2)(x,t)dx+\frac{1}{2}\int_{\mathbb{T}^d} (\varrho_0|\uu_0|^2)(x)dx+\frac{1}{2} \int_{0}^{t}\int_{\mathbb{T}^d} \nabla \cdot (\varrho \uu)^\varepsilon\frac{|(\varrho \uu)^\varepsilon|^2}{(\varrho^\varepsilon)^2}dxds\right| = 0.
%\end{equation}

\subsection{Estimate of $(B)$} By integration by parts,
\begin{align*}
(B)&= - \int_{\tau}^{t}\int_{\mathbb{T}^d}\nabla \frac{(\varrho \uu)^\varepsilon}{\varrho^\varepsilon}(\varrho \uu \otimes\uu)^\varepsilon dxds\\
&= -\int_{\tau}^{t}\int_{\mathbb{T}^d}\nabla  \frac{(\varrho \uu)^\varepsilon}{\varrho^\varepsilon}[(\varrho \uu \otimes\uu)^\varepsilon-(\varrho \uu)^\varepsilon \otimes\uu^\varepsilon]dxds - \int_\tau^t\int_{\T^d}\na \frac{(\vr\uu)^\eps}{\vre}((\vr\uu)^\eps\otimes\ue)dxds\\
&=: (B1) +  \int_{\tau}^{t}\int_{\mathbb{T}^d}\frac{(\varrho \uu)^\varepsilon}{\varrho^\varepsilon}\nabla \cdot((\varrho \uu)^\varepsilon \otimes\uu^\varepsilon)dxds \\
&= (B1) +  \int_{\tau}^{t}\int_{\mathbb{T}^d}\nabla \cdot\uu^\varepsilon \frac{|(\varrho\uu)^\varepsilon|^2}{\varrho^\varepsilon}dxds+\frac{1}{2} \int_{\tau}^{t}\int_{\mathbb{T}^d} \frac{\uu^\varepsilon}{\varrho^\varepsilon} \nabla |(\varrho\uu)^\varepsilon|^2dxds \\
&= (B1) + \frac{1}{2}\int_{\tau}^{t}\int_{\mathbb{T}^d}\nabla \cdot \uu^\varepsilon\frac{|(\varrho\uu)^\varepsilon|^2}{\varrho^\varepsilon}dxds
-\frac{1}{2} \int_{\tau}^{t}\int_{\mathbb{T}^d} \uu^\varepsilon \nabla (\frac{1}{\varrho^\varepsilon})  |(\varrho\uu)^\varepsilon|^2dxds
\\&= (B1) + \frac{1}{2}\int_{\tau}^{t}\int_{\mathbb{T}^d}\nabla \cdot  (\varrho^\varepsilon \uu^\varepsilon) \frac{|(\varrho\uu)^\varepsilon|^2}{(\varrho^\varepsilon)^2}dxds\\
&= (B1) + \frac{1}{2}\int_{\tau}^{t}\int_{\mathbb{T}^d}\nabla \cdot [ (\varrho^\varepsilon \uu^\varepsilon) - (\vr\uu)^\eps] \frac{|(\varrho\uu)^\varepsilon|^2}{(\varrho^\varepsilon)^2}dxds + \frac 12 \int_{\tau}^{t}\int_{\mathbb{T}^d} \nabla \cdot (\varrho \uu)^\varepsilon\frac{|(\varrho \uu)^\varepsilon|^2}{(\varrho^\varepsilon)^2}dxds\\
&=: (B1) + (B2) + (B3).
\end{align*}
It is obvious that $(A2) + (B3) = 0$. We now estimate the remainders $(B1)$ and $(B2)$. By H\"older's inequality
\begin{equation} \label{M1} \begin{aligned}
|(B1)| &= \left|\int_{\tau}^{t}\int_{\mathbb{T}^d}\nabla \frac{(\varrho \uu)^\varepsilon}{\varrho^\varepsilon}[(\varrho \uu \otimes\uu)^\varepsilon-(\varrho \uu)^\varepsilon \otimes\uu^\varepsilon]dxds\right|\\
&\leq \left\|\nabla \frac{(\varrho \uu)^\varepsilon}{\varrho^\varepsilon}\right\|_{L^{4}(\T^d \times (0,T))} \|(\varrho \uu \otimes\uu)^\varepsilon-(\varrho \uu)^\varepsilon \otimes\uu^\varepsilon\|_{L^{\frac{4}{3}}(\T^d \times (0,T))}.
\end{aligned} \end{equation}
From Lemma \ref{GN} (for $d=2$) and the assumption \eqref{rhod=3} (for $d=3$), we see that $\uu\in L^4(\T^d \times (0,T) )$. Moreover, by assumption, $\uu \in L^2(0,T;H^1(\T^d))$. Therefore, we can apply Lemmas \ref{lemma1} (iii) and \ref{keyla} to obtain
\begin{equation} \label{X6''}
\limsup_{\varepsilon \to 0}\varepsilon\left\|\nabla \cdot \frac{(\varrho \uu)^\varepsilon}{\varrho^\varepsilon}\right\|_{L^{4}(\T^d \times (0,T))}=0,
\end{equation}
\begin{align}\label{X6'}
\limsup_{\varepsilon \to 0}\varepsilon^{-1}\|(\varrho \uu \otimes\uu)^\varepsilon-(\varrho \uu)^\varepsilon \otimes\uu^\varepsilon\|_{L^{\frac{4}{3}}(\T^d \times (0,T))}=0,
\end{align}
which yields 
\begin{equation}\label{B1}
	\limsup_{\eps\to 0}\limsup_{\tau \to 0}|(B1)| = 0.
\end{equation}

For $(B2)$ we first use integration by parts to have
\begin{equation} \label{Y1ss}  \begin{aligned}
|(B2)| &= \frac 12\left|\int_{\tau}^{t}\int_{\mathbb{T}^d}[(\varrho^\varepsilon \uu^\varepsilon) - (\vr\uu)^\eps]\nabla \cdot  \frac{|(\varrho\uu)^\varepsilon|^2}{(\varrho^\varepsilon)^2}dxds\right|\\
&\leq C \|(\varrho \uu)^\varepsilon-\varrho^\varepsilon\uu^\varepsilon\|_{L^2(\T^d \times (0,T))}\left\|\nabla \cdot \frac{|(\varrho \uu)^\varepsilon|^2}{(\varrho^\varepsilon)^2}\right\|_{L^2(\T^d \times (0,T))}.
\end{aligned} \end{equation}
One can justify from H\"older's inequality and Lemma \ref{lemma1} that
\begin{equation} \label{rem}
	\left\|\nabla \cdot \frac{|(\varrho \uu)^\varepsilon|^2}{(\varrho^\varepsilon)^2}\right\|_{L^2(\T^d \times (0,T))} \leq  C\eps^{-1}\|\uu\|_{L^4(\T^d \times (0,T))}^2.
\end{equation}
On the other hand,  it follows from Lemma \ref{keyla} that
\begin{align}\label{B2_0}
\limsup_{\varepsilon \to 0}\varepsilon^{-1}\|(\varrho \uu)^\varepsilon-\varrho^\varepsilon\uu^\varepsilon\|_{L^2(\T^d \times (0,T))}=0.
\end{align}
From \eqref{Y1ss} -- \eqref{B2_0}, we obtain
\begin{equation}\label{B2}
	\limsup_{\eps\to 0}\limsup_{\tau \to 0}|(B2)| = 0.
\end{equation} 

\subsection{Estimate of $(C)$} 
Integrating by parts and using \eqref{Trho_approx} lead to
\begin{equation} \label{C-a1} \begin{aligned}
	(C) &=  \int_{\tau}^{t}\int_{\mathbb{T}^d} \frac{(\varrho \uu)^\varepsilon}{\varrho^\varepsilon}\nabla [(p(\vr))^\varepsilon - p(\vre)] dxds + \int_{\tau}^{t}\int_{\T^d}\frac{(\varrho \uu)^\varepsilon}{\varrho^\varepsilon}\na p(\vre) dxds\\
	&=: (C1) - \int_{\tau}^{t}\int_{\mathbb{T}^d} \nabla \cdot (\varrho \uu)^\varepsilon \GG(\vre) dxds \\
	&= (C1) + \int_{\tau}^{t}\int_{\mathbb{T}^d} \partial_t\varrho^\varepsilon \GG(\vre) dxds,
\end{aligned} \end{equation}
where $\mathcal G$ is defined as
\begin{equation}\label{G}
\GG(z): = \int_1^z\frac{p'(\xi)}{\xi}d\xi.
\end{equation}
{\color{black} Recalling that $\HH$ is defined in \eqref{GH}, we can easily compute that
\begin{equation} \label{PG}
\partial_t \HH(\varrho^\varepsilon) = \partial_t \varrho^\varepsilon \left( \GG(\varrho^\varepsilon) + p(1)  \right), 
\end{equation}
which, together with \eqref{Trho_approx}, implies
\begin{equation} \label{C-a2} \begin{aligned}
\int_{\T^d} \partial_t \varrho^\varepsilon \GG(\varrho^\varepsilon) dx &= \int_{\T^d} \partial_t \HH(\varrho^\varepsilon)dx - p(1) \int_{\T^d} \partial_t \varrho^\varepsilon dx \\
&= \int_{\T^d} \partial_t \HH(\varrho^\varepsilon)dx + p(1) \int_{\T^d} \nabla \cdot (\varrho \uu)^\varepsilon dx \\
&= \int_{\T^d} \partial_t \HH(\varrho^\varepsilon)dx.
\end{aligned}  \end{equation}          
From \eqref{C-a1} and \eqref{C-a2}, we obtain
\begin{align*}
(C)=(C1) + \int_{\tau}^{t}\int_{\T^d}\pa_t \HH(\vre) dxds=: (C1) + (C2).
\end{align*}}

Since $(C2)$ is a desired term, we only need to estimate $(C1)$. 

We first consider the case $d=2$. We see that
\begin{equation} \label{X1} \begin{aligned}
|(C1)|&= \left|\int_{\tau}^{t}\int_{\mathbb{T}^2} \frac{(\varrho \uu)^\varepsilon}{\varrho^\varepsilon}\nabla [(p(\vr))^\varepsilon - p(\vre)]dxds\right|\\
&\leq \left|  \int_\tau^t \int_{\T^2} \frac{ (\varrho \uu)^\varepsilon - \varrho^\varepsilon \uu^\varepsilon  }{\varrho^\varepsilon} \nabla [(p(\vr))^\varepsilon - p(\vre)]  dxds \right|  + \left|  \int_\tau^t \int_{\T^2}   \nabla \cdot \uu^\varepsilon [(p(\vr))^\varepsilon - p(\vre)] dxds \right| \\
&\leq  \left\| \frac{ (\varrho \uu)^\varepsilon - \varrho^\varepsilon \uu^\varepsilon  }{\varrho^\varepsilon}\right\|_{L^1(\T^2 \times (0,T))} \blue{\| \nabla  [(p(\vr))^\varepsilon - p(\vre)]  \|_{L^\infty(\T^2 \times (0,T))}} \\
&\quad + \| \nabla \cdot \uu^\varepsilon\|_{L^1(0,T;L^\infty(\mathbb{T}^2))} \blue{\|(p(\vr))^\varepsilon-p(\vre)\|_{L^\infty(0,T;L^1(\T^2))}}.
\end{aligned} \end{equation}

We now deal with the first term on the right-hand side of \eqref{X1}. By Lemma \ref{lemma1} (i), we derive
\begin{align*}   
\blue{\| \nabla  [(p(\vr))^\varepsilon - p(\vre)]  \|_{L^\infty(\T^2 \times (0,T))} \leq \|\nabla (p(\vr))^\eps\|_{L^\infty(\T^2 \times (0,T))} + \|p'(\vr)\nabla \vre\|_{L^\infty(\T^2 \times (0,T))} \leq C\varepsilon^{-1}}
\end{align*}
thanks to $\varrho \in L^\infty(\T^2 \times (0,T))$.  Since $\uu \in L^2(0,T;H^1(\T^2))$, by Lemma \ref{keyla},
$$\limsup_{\varepsilon \to 0} \varepsilon^{-1}\left\| \frac{ (\varrho \uu)^\varepsilon - \varrho^\varepsilon \uu^\varepsilon  }{\varrho^\varepsilon}\right\|_{L^1(\T^2 \times (0,T))} \leq C \limsup_{\varepsilon \to 0} \varepsilon^{-1} \| (\varrho \uu)^\varepsilon - \varrho^\varepsilon \uu^\varepsilon \|_{L^1(\T^2 \times (0,T))} = 0.
$$
It follows that
\begin{equation} \label{X1-a}
\limsup_{\varepsilon \to 0} \left\| \frac{ (\varrho \uu)^\varepsilon - \varrho^\varepsilon \uu^\varepsilon  }{\varrho^\varepsilon} \right\|_{L^1(\T^2 \times (0,T))} \| \nabla  [(p(\vr))^\varepsilon - p(\varrho^\varepsilon)]  \|_{L^\infty(\T^2 \times (0,T))} =0.
\end{equation}
For the second term on the right hand side of \eqref{X1}, we first show that
\begin{equation} \label{L1-1} \blue{\|(p(\vr))^\varepsilon-p(\varrho^\varepsilon)\|_{L^\infty(0,T;L^1(\T^2))}\leq C\varepsilon}. 
\end{equation}
\blue{
In order to prove \eqref{L1-1}, we first use Taylor's expansion to estimate
%{\color{red} Indeed, in order to prove \eqref{L1-1}, we will use the following claim. \medskip
\begin{equation} \label{p}
|p(a+b) - p(a) - p'(a)b| = |p''(\xi)||b|^2 \leq \Lambda |b|^2,
\end{equation}	
for any $a,b$ such that $a, a+b \in (c_1, c_2)$ (with $c_1, c_2$ are in \eqref{Condition}) and $\Lambda = \sup_{\xi\in [c_1,c_2]}|p''(\xi)|$. Now, by writing 
\begin{equation*}
	\vr(x-y,s) = \vr(x,s) + [\vr(x-y,s) - \vr(x,s)]
\end{equation*}
and 
\begin{equation*}
	\int_{\T^2}\vr(x-y,s)\omega_\eps(y)dy = \vr(x,s) + \int_{\T^2}(\vr(x-y,s) - \vr(x,s))\omega_\eps(y)dy
\end{equation*}
we can apply \eqref{p} with $a = \vr(x,s)$ and $b = \vr(x-y,s) - \vr(x,s)$, and then $a = \vr(x,s)$ and $b = \int_{\T^2}(\vr(x-y,s)-\vr(x,s))\omega_\varepsilon(y)dy$ respectively (note that $a,a+b \in (c_1,c_2)$ thanks to \eqref{Condition}), to obtain
\begin{equation} \label{X3''}
\begin{aligned}
&|(p(\vr))^\eps(x,s) - p(\vre)(x,s)|\\
&\leq \left|\int_{\T^2}[p(\vr(x-y,s)) - p(\vr(x,s)) - p'(\vr(x,s))(\vr(x-y,s) - \vr(x,s))]\omega_\eps(y)dy\right|\\
&\quad  + \left|p\left(\int_{\T^2}\vr(x-y,s)\omega_\eps(y)du\right) - p(\vr(x,s)) - p'(\vr(s,x))\int_{\T^2}(\vr(x-y,s) - \vr(x,s))\omega_\eps(y)dy\right| \\
&\leq 2\Lambda \int_{\T^2}|\vr(x-y,s) - \vr(x,s)|^2\omega_\varepsilon(y)dy.
\end{aligned}
\end{equation}
This, together with assumption \eqref{rhod=2}, implies 
\begin{align*}
\|(p(\vr))^\varepsilon(\cdot,s)-p(\varrho^\varepsilon)(\cdot,s)\|_{L^1(\T^2)} &\leq 2\Lambda  \int_{\T^2}  \int_{\T^2}|\varrho(x-y,s) - \varrho(x,s)|^2\omega_\varepsilon(y)dydx  \\
&\leq 2\Lambda \varepsilon \sup_{|h|<\varepsilon}|h|^{-1}\| \varrho(\cdot + h,s) - \varrho(\cdot,s)\|_{L^2(\T^2)}^2 \\
&\leq 2\Lambda C \varepsilon.
\end{align*}
This yields \eqref{L1-1}.}

Thus the second term on the right hand side of \eqref{X1} is estimated as
\begin{equation}\label{X1-b}
	\|\na \cdot \ue\|_{L^1(0,T;L^\infty(\T^2))}\|(p(\vr))^\eps - p(\vre)\|_{L^\infty(0,T;L^1(\T^2))} \leq C\varepsilon\|\na\cdot \ue\|_{L^1(0,T;L^\infty(\T^2))}.
\end{equation}
By Lemma \ref{lemma1} (iv), 
$$
\lim_{\varepsilon \to 0} \varepsilon\|\na\cdot \ue\|_{L^1(0,T;L^\infty(\T^2))} =0.
$$
This and \eqref{X1-b} ensure
\begin{equation}\label{X1-b1}
	\limsup_{\eps \to 0}\|\na \cdot \ue\|_{L^1(0,T;L^\infty(\T^2))}\|p(\vre) - (p(\vr))^\eps\|_{L^\infty(0,T;L^1(\T^2))}  = 0.
\end{equation}
We conclude from \eqref{X1}, \eqref{X1-a} and \eqref{X1-b1} that in the case $d=2$,
\begin{equation}\label{C1-d2}
	\limsup_{\eps \to 0}\limsup_{\tau \to 0}|(C1)| = 0.
\end{equation}

For the case $d = 3$ we estimate by using integration by parts and H\"older's inequality that
\begin{equation} \label{K1} \begin{aligned}
|(C1)|&= \left|\int_{0}^{t}\int_{\mathbb{T}^3} \na \frac{(\varrho \uu)^\varepsilon}{\varrho^\varepsilon}[(p(\vr))^\varepsilon - p(\vre)] dxds\right|\\
&\leq  \left\| \nabla \cdot \frac{(\varrho \uu)^\varepsilon}{\varrho^\varepsilon}\right\|_{L^1(0,T;L^6(\mathbb{T}^3))} \blue{\|[(p(\vr))^\varepsilon - p(\vre)]\|_{L^\infty(0,T;L^\frac{6}{5}(\T^3)) }}.
\end{aligned} \end{equation}
From \eqref{X3''} (note that this estimate is independent of dimensions), H\"older's inequality and assumption \eqref{rhod=3}, 
we have
\begin{align*}
\|(p(\vr))^\varepsilon(\cdot,s)-p(\varrho^\varepsilon)(\cdot,s)\|_{L^{\frac{6}{5}}(\T^3)} &\leq C \left( \int_{\T^3}  \int_{\T^3}|\varrho(x-y,s) - \varrho(x,s)|^{\frac{12}{5}}\omega_\varepsilon(y)dydx \right)^{\frac{5}{6}}  \\
&\leq C \varepsilon \sup_{|h|<\varepsilon}|h|^{-1}\| \varrho(\cdot + h,s) - \varrho(\cdot,s)\|_{L^{\frac{12}{5}}(\T^3)}^2 \\
&\leq C \varepsilon.
\end{align*}
This implies
\begin{equation}\label{K3}
	\|(p(\vr))^\eps - p(\vre)\|_{L^\infty(0,T;L^{\frac 65}(\T^3))} \leq C\varepsilon.
\end{equation} 
Inserting \eqref{K3} into \eqref{K1} one has
\begin{equation}
	|(C1)| \leq C \varepsilon \left\| \nabla \cdot \frac{(\varrho \uu)^\varepsilon}{\varrho^\varepsilon}\right\|_{L^1(0,T;L^6(\mathbb{T}^3))}.
\end{equation}
Since $\varrho, \varrho^{-1} \in L^\infty(\T^3 \times (0,T))$ and $\uu \in L^2(0,T;H^1(\T^3)) \subset L^1(0,T;L^6(\T^3))$, by using Lemma \ref{lemma1} (iii), we deduce
\begin{equation*}
	\lim_{\varepsilon \to 0} \varepsilon \left\| \nabla \cdot \frac{(\varrho \uu)^\varepsilon}{\varrho^\varepsilon}\right\|_{L^1(0,T;L^6(\mathbb{T}^3))}=0.
\end{equation*}
Therefore, in the case $d = 3$,
\begin{equation}\label{C1-d3}
	\limsup_{\eps\to 0}\limsup_{\tau\to 0}|(C1)| = 0.
\end{equation}

\subsection{Estimate of $(D)$.} \blue{It is easy to see that
\begin{align*}
(D)  =   \int_{\tau}^{t}\int_{\mathbb{T}^d} \nabla \cdot (\nu(\varrho) \mathbb{D}\uu)^\varepsilon \uu^\varepsilon dxds -\int_{\tau}^{t}\int_{\mathbb{T}^d} \nabla \cdot (\nu(\varrho) \mathbb{D}\uu)^\varepsilon \frac{(\varrho \uu)^\varepsilon-\varrho^\varepsilon \uu^\varepsilon}{\varrho^\varepsilon} dxds.
\end{align*}
By  H\"older's inequality, Lemma \ref{lemma1} (ii), the fact that $\varrho \in L^\infty(\T^d \times (0,T))$ and $\nu \in C(\R_+)$,  we deduce
\begin{equation}\label{D0-1}
\begin{aligned}
\left|\int_{\tau}^{t}\int_{\mathbb{T}^d} \nabla \cdot (\nu(\varrho) \mathbb{D}\uu)^\varepsilon \frac{(\varrho \uu)^\varepsilon-\varrho^\varepsilon \uu^\varepsilon}{\varrho^\varepsilon} dxds\right| &\leq  \| \nabla \cdot (\nu(\varrho) \D \uu)^\varepsilon \|_{L^2(\T^d \times (0,T))} \| \frac{  (\varrho \uu)^\varepsilon-\varrho^\varepsilon \uu^\varepsilon }{\varrho^\varepsilon} \|_{L^2(\T^d \times (0,T))} \\
&\leq C \varepsilon^{-1}  \|  \nu(\varrho) \D \uu \|_{L^2(\T^d \times (0,T))} \|  (\varrho \uu)^\varepsilon-\varrho^\varepsilon \uu^\varepsilon \|_{L^2(\T^d \times (0,T))} \\
& \leq C \varepsilon^{-1}  \| \nabla \uu \|_{L^2(\T^d \times (0,T))} \| (\varrho \uu)^\varepsilon-\varrho^\varepsilon \uu^\varepsilon \|_{L^2(\T^d \times (0,T))}.
\end{aligned}
\end{equation}
Since $\nabla \uu\in L^2(\T^d \times (0,T))$ and $\varrho\in L^\infty(\T^d \times (0,T))$, by Lemma \ref{keyla} we get
\begin{equation} \label{eD0}
\limsup_{\varepsilon \to 0}\varepsilon^{-1}\|(\varrho \uu)^\varepsilon-\varrho^\varepsilon \uu^\varepsilon\|_{L^2(\T^d \times (0,T))}=0.
\end{equation}
Therefore,
\begin{equation}\label{for-ex1}
	\limsup_{\eps\to 0}\limsup_{\tau\to 0}\left|\int_{\tau}^{t}\int_{\mathbb{T}^d} \nabla \cdot (\nu(\varrho) \mathbb{D}\uu)^\varepsilon \frac{(\varrho \uu)^\varepsilon-\varrho^\varepsilon \uu^\varepsilon}{\varrho^\varepsilon} dxds\right| = 0.
\end{equation}}
\subsection{Estimate of $(E)$}\label{EstimateE}
\blue{We write
\begin{align*}
(E) &= - \int_{\tau}^t\int_{\T^d}\na (\mu(\vr)\di \uu)^\eps \ue dxds - \int_{\tau}^t\int_{\T^d}\na(\mu(\vr)\di \uu)^\eps \frac{(\vr\uu)^\eps - \vre \ue}{\vre}dxds\\
&= \int_\tau^t\int_{\T^d}(\mu(\vr)\di \uu)^\eps (\di \ue)dxds -  \int_{\tau}^t\int_{\T^d}\na(\mu(\vr)\di \uu)^\eps \frac{(\vr\uu)^\eps - \vre \ue}{\vre}dxds.
\end{align*}
Similarly to \eqref{for-ex1} we have
\begin{equation}
	\limsup_{\eps\to 0}\limsup_{\tau \to 0}\left|\int_{\tau}^t\int_{\T^d}\na(\mu(\vr)\di \uu)^\eps \frac{(\vr\uu)^\eps - \vre \ue}{\vre}dxds\right| = 0.
\end{equation}}

\subsection{Conclusion of the Proof of Theorem \ref{cNS-torus}}\label{sub:conclusion1}
Collecting all the above estimates and putting them into \eqref{multiply} we obtain
\begin{equation}
\begin{aligned}
\limsup_{\eps\to 0}\limsup_{\tau \to 0}\biggl|\int_{\tau}^{t}\int_{\T^d}\pa_t\left[\frac 12 \frac{|(\vr\uu)^\eps|^2}{\vre} + \HH(\vre)\right]dxds&- \int_{\tau}^{t}\int_{\mathbb{T}^d} \nabla \cdot (\nu(\varrho) \mathbb{D}\uu)^\varepsilon \uu^\varepsilon dxds\\
&-\int_{\tau}^{t}\int_{\mathbb{T}^d} (\mu(\vr)\di\uu)^\eps (\di \ue) dxds \biggr| = 0.
\end{aligned}
\end{equation}
Using the weak continuity of $\vr$ and $\vr\uu$ in \eqref{ini1} and \eqref{ini2}, and the limits
\blue{
\begin{align*}
\limsup_{\varepsilon \to 0}\limsup_{\tau \to 0}\left|\int_{\tau}^{t}\int_{\mathbb{T}^d}  \nabla \cdot (\nu(\varrho) \mathbb{D}\uu)^\varepsilon \uu^\varepsilon dxds + \int_0^t\int_{\T^d}\nu(\varrho)|\mathbb D\uu|^2dxds\right| = 0,
\end{align*}}
\blue{
\begin{align*}
\limsup_{\varepsilon \to 0}\limsup_{\tau \to 0}\left|\int_{\tau}^{t}\int_{\mathbb{T}^d}  (\mu(\vr)\di\uu)^\eps (\di \ue) dxds + \int_0^t\int_{\T^d}\mu(\varrho)|\di\uu|^2dxds\right| = 0,
\end{align*}}
we can finally conclude the proof of Theorem \ref{cNS-torus}. \qed

\section{Proof of Theorem \ref{cNS-domain}}\label{sec:thm2}

The proof of Theorem \ref{cNS-domain} is similar to that of Theorem \ref{cNS-torus}, except for the fact that we have to take care of the boundary layers when integrating by parts. More precisely, by smoothing \eqref{NS} we obtain
\begin{equation}\label{Drho_approx}
	\pa_t\vre + \na\cdot(\vr\uu)^\eps = 0, \quad \text{ in } \Omega_\eps,
\end{equation}
and
\begin{equation}\label{Du_approx}
	\pa_t(\vr\uu)^\eps + \na\cdot(\vr\uu\otimes\uu)^\eps + \na(p(\vr))^\eps - \di(\nu(\vr)\mathbb D\uu)^\eps - \nabla(\mu(\vr)\di\uu)^\eps = 0 \quad \text{ in } \Omega_\eps,
\end{equation}
recalling $\Omega_\varepsilon = \{x\in\Omega: d(x, \partial\Omega) >\varepsilon\}$. Take $0 < \eps < \eps_1/10 < \eps_2/10 < r_0/100$ we obtain by multiplying \eqref{Du_approx} by $(\vr\uu)^\eps/\vre$ and integrating on $(\tau, t)\times \Omega_{\eps_2}$ with $0 < \tau < t < T$ 
\begin{equation}\label{Y1}
\begin{aligned}
	&\int_{\tau}^{t}\int_{\Omega_{\eps_2}}  \frac{(\varrho \uu)^\varepsilon}{\varrho^\varepsilon}\partial_t(\varrho \uu)^\varepsilon dxds + \int_{\tau}^{t}\int_{\Omega_{\eps_2}} \frac{(\varrho \uu)^\varepsilon}{\varrho^\varepsilon}\nabla \cdot (\varrho \uu \otimes\uu)^\varepsilon dxds + \int_{\tau}^{t}\int_{\Omega_{\eps_2}} \frac{(\varrho \uu)^\varepsilon}{\varrho^\varepsilon}\nabla (p(\varrho))^\varepsilon dxds \\
	&-  \int_{\tau}^{t}\int_{\Omega_{\eps_2}} \nabla \cdot (\nu(\varrho) \mathbb{D}\uu)^\varepsilon \frac{(\varrho \uu)^\varepsilon}{\varrho^\varepsilon} dxds - \int_{\tau}^{t}\int_{\Omega_{\eps_2}} \nabla(\mu(\vr)\di\uu)^\eps \frac{(\varrho \uu)^\varepsilon}{\varrho^\varepsilon} dxds = 0.	
\end{aligned}
\end{equation}
Taking $\eps_3 > 0$ to be small, we integrate \eqref{Y1} with respect to $\eps_2$ on $(\eps_1, \eps_1+\eps_3)$ to get
\begin{equation}\label{Y2}
	\begin{aligned}
	&\frac{1}{\eps_3}\int_{\eps_1}^{\eps_1+\eps_3}\int_{\tau}^{t}\int_{\Omega_{\eps_2}}  \frac{(\varrho \uu)^\varepsilon}{\varrho^\varepsilon}\partial_t(\varrho \uu)^\varepsilon dxdsd\eps_2 + \frac{1}{\eps_3}\int_{\eps_1}^{\eps_1+\eps_3}\int_{\tau}^{t}\int_{\Omega_{\eps_2}} \frac{(\varrho \uu)^\varepsilon}{\varrho^\varepsilon}\nabla \cdot (\varrho \uu \otimes\uu)^\varepsilon dxdsd\eps_2\\
	& \blue{+ \frac{1}{\eps_3}\int_{\eps_1}^{\eps_1+\eps_3}\int_{\tau}^{t}\int_{\Omega_{\eps_2}} \frac{(\varrho \uu)^\varepsilon}{\varrho^\varepsilon}\nabla (p(\varrho))^\varepsilon dxdsd\eps_2 -  \frac{1}{\eps_3}\int_{\eps_1}^{\eps_1+\eps_3}\int_{\tau}^{t}\int_{\Omega_{\eps_2}} \nabla \cdot (\nu(\vr) \mathbb{D}\uu)^\varepsilon \frac{(\varrho \uu)^\varepsilon}{\varrho^\varepsilon} dxdsd\eps_2}\\
	&-  \blue{\frac{1}{\eps_3}\int_{\eps_1}^{\eps_1+\eps_3}\int_{\tau}^{t}\int_{\Omega_{\eps_2}} \nabla(\mu(\vr)\di \uu)^\eps \frac{(\varrho \uu)^\varepsilon}{\varrho^\varepsilon} dxdsd\eps_2= 0.}
	\end{aligned}
\end{equation}
We denote by $(F)$, $(G)$, $(H)$, $(K)$ and $(L)$ the first, second, third, forth and fifth terms on the left hand side of \eqref{Y2} respectively. We will estimate them separately in the following subsections. 
\subsection{Estimate of $(F)$}
This term is estimated similarly to $(A)$,
\begin{align*}
(F)&=\frac{1}{2} \frac{1}{\eps_3}\int_{\eps_1}^{\eps_1+\eps_3}\int_{\tau}^{t}\int_{\Omega_{\eps_2}} \partial_t\left(\frac{|(\varrho \uu)^\varepsilon|^2}{\varrho^\varepsilon}\right)dxdsd\eps_2 + \frac{1}{2}\frac{1}{\eps_3}\int_{\eps_1}^{\eps_1+\eps_3} \int_{\tau}^{t}\int_{\Omega_{\eps_2}} \partial_t \varrho^\varepsilon\frac{|(\varrho \uu)^\varepsilon|^2}{(\varrho^\varepsilon)^2}dxdsd\eps_2 \\
&= \frac{1}{2} \frac{1}{\eps_3}\int_{\eps_1}^{\eps_1+\eps_3}\int_{\tau}^{t}\int_{\Omega_{\eps_2}} \partial_t\left(\frac{|(\varrho \uu)^\varepsilon|^2}{\varrho^\varepsilon}\right)dxdsd\eps_2-\frac{1}{2}\frac{1}{\eps_3}\int_{\eps_1}^{\eps_1+\eps_3} \int_{\tau}^{t}\int_{\Omega_{\eps_2}} \nabla \cdot (\varrho \uu)^\varepsilon\frac{|(\varrho \uu)^\varepsilon|^2}{(\varrho^\varepsilon)^2}dxdsd\eps_2\\
&=: (F1) + (F2).
\end{align*}
The term $(F1)$ is desired, while $(F2)$ will be canceled by $(G4)$ later.
\subsection{Estimate of $(G)$}
We estimate $(G)$ similarly to $(B)$ by using integration by parts
\begin{align*}
(G)&= \thii\frac{(\varrho \uu)^\varepsilon}{\varrho^\varepsilon}\nabla \cdot [(\vr\uu\otimes\uu)^\eps - (\vr\uu)^\eps\otimes\ue] dxdsd\eps_2\\
&\quad + \thii \frac{(\vr\uu)^\eps}{\vre}\di((\vr\uu)^\eps\otimes\ue)\dxte\\
&=: (G1) + \thii \di\ue \frac{|(\vr\uu)^\eps|^2}{\vre}\dxte\\
&\quad + \frac 12 \thii \frac{\ue}{\vre}\na|(\vr\uu)^\eps|^2\dxte\\
&=: (G1) + \frac 12\thiip |(\vr\uu)^\eps|^2\frac{\ue}{\vre}n(\theta)\dHte\\
&\quad + \frac 12\thii \di[\vre\ue - (\vr\uu)^\eps]\frac{|(\vr\uu)^\eps|^2}{(\vre)^2}\dxte\\
&\quad + \frac 12 \thii \di(\vr\uu)^\eps\frac{|(\vr\uu)^\eps|^2}{(\vre)^2}\dxte\\
&=: (G1) + (G2)^{\mathrm{bdr}} + (G3) + (G4).
\end{align*}
The superscript $``\mathrm{bdr}"$ in $(G2)^{\mathrm{bdr}}$ means that this term contains a boundary layer. It's obvious that $(G4) + (F2) = 0$. The term $(G1)$ is estimated using integration by parts as
\begin{align*}
	(G1) &= -\thii \di \frac{(\vr\uu)^\eps}{\vre}[(\vr\uu\otimes\uu)^\eps - (\vr\uu)^\eps\otimes\ue]\dxte\\
	&\quad + \thiip [(\vr\uu\otimes\uu)^\eps - (\vr\uu)^\eps\otimes\ue]\frac{(\vr\uu)^\eps}{\vre}n(\theta)\dHte\\
	&=: (G11) + (G12)^{\mathrm{bdr}}.
\end{align*}
We estimate $(G11)$ similarly to $(B1)$ in \eqref{M1}--\eqref{X6'} and therefore
\begin{equation}
	\limsup_{\eps\to 0}\limsup_{\tau\to 0}|(G11)| = 0.
\end{equation}

The term $ (G12)^{\mathrm{bdr}}$ will be treated later on, together with other boundary terms.

For $(G3)$ it follows from integration by parts that
\begin{align*}
	(G3)&= -\frac 12\thii [\vre\ue - (\vr\uu)^\eps]\na\frac{|(\vr\uu)^\eps|^2}{(\vre)^2}\dxte\\
	&\quad + \frac 12\thiip [\vre\ue - (\vr\uu)^\eps]n(\theta)\frac{|(\vr\uu)^\eps|^2}{(\vre)^2}\dHte\\
	&=: (G31) + (G32)^{\mathrm{bdr}}.
\end{align*}
The term $(G31)$ is estimated similarly to $(B2)$ in \eqref{Y1ss}--\eqref{B2_0} and therefore
\begin{equation}
\limsup_{\eps\to 0}\limsup_{\tau\to 0}|(G31)| = 0.
\end{equation}

It remains to estimate the boundary terms $(G12)^{\mathrm{bdr}}$, $(G2)^{\mathrm{bdr}}$  and $(G32)^{\mathrm{bdr}}$. As for the term $(G12)^{\mathrm{bdr}}$, we use the coarea formula and then letting successively $\tau \to 0$ and $\varepsilon \to 0$ to obtain
%\begin{align*}
%|(F12)^{\mathrm{bdr}}| &\leq \frac{1}{\varepsilon_3}\int_{\tau}^t \int_{\Omega_{\varepsilon_1} \setminus \Omega_{\varepsilon_1 + \varepsilon_3}} |(\varrho \uu \otimes \uu)^\varepsilon - (\varrho \uu)^\varepsilon \otimes \uu^\varepsilon| \left|  \frac{(\varrho \uu)^\varepsilon}{\varrho^\varepsilon} \right| dxds \\
%&\leq  \frac{1}{\varepsilon_3} \| (\varrho \uu \otimes \uu)^\varepsilon - (\varrho \uu)^\varepsilon \otimes \uu^\varepsilon  \|_{L^\frac{4}{3}((\Omega_{\varepsilon_1} \setminus \Omega_{\varepsilon_1 + \varepsilon_3}) \times (0,T) )} \left\|  \frac{(\varrho \uu)^\varepsilon}{\varrho^\varepsilon} \right\|_{L^4((\Omega_{\varepsilon_1} \setminus \Omega_{\varepsilon_1 + \varepsilon_3}) \times (0,T) )} \\
%&\leq C \frac{\varepsilon}{\varepsilon_3} \| \uu \|_{L^4((\Omega \setminus \Omega_{2(\varepsilon_1 + \varepsilon_3)}) \times (0,T) )}^2 \| \uu \|_{L^2(0,T,H^1(\Omega \setminus \Omega_{2(\varepsilon_1 + \varepsilon_3)}))}.
%\end{align*}
%As a consequence,
\begin{equation}
	\limsup_{\eps\to 0}\limsup_{\tau\to 0}|(G12)^{\mathrm{bdr}}|  = 0.
\end{equation}

Next, we deal with the term $(G2)^{\mathrm{bdr}}$. By using the coarea formula, H\"older's inequality and Lemma \ref{Poincare}, we get
\begin{align*}
	\limsup_{\varepsilon_1, \varepsilon \to 0} \limsup_{\tau \to 0}|(G2)^{\mathrm{bdr}}| &= \limsup_{\varepsilon_1, \varepsilon \to 0} \limsup_{\tau \to 0} \left| \frac{1}{2\varepsilon_3} \int_\tau^t \int_{\Omega_{\varepsilon_1} \setminus \Omega_{\varepsilon_1 + \varepsilon_3} } |(\varrho \uu)^\varepsilon|^2 \frac{\uu^\varepsilon}{\varrho^\varepsilon} n(x) dx ds \right| \\
	&\leq \frac{C}{2\eps_3} \int_0^T\int_{\Omega \backslash\Omega_{\eps_3}}|\uu|^3dxds \\
	&\leq \frac{C}{2\eps_3}  \| \uu \|_{L^4((\Omega \setminus \Omega_{\varepsilon_3}) \times (0,T)) }^2 \| \uu  \|_{L^2((\Omega \setminus \Omega_{\varepsilon_3}) \times (0,T))} \\
	&\leq C  \| \uu \|_{L^4((\Omega \setminus \Omega_{\varepsilon_3}) \times (0,T))}^2 \| \nabla \uu  \|_{L^2((\Omega \setminus \Omega_{2\varepsilon_3} \times (0,T))}.
\end{align*}
Since $\uu \in L^4(\Omega \times (0,T))$ and $\uu \in L^2(0,T;H^1(\Omega))$, by letting $\varepsilon_3 \to 0$, we obtain  
\begin{equation*}
	\limsup_{\varepsilon_3 \to 0}\limsup_{\varepsilon_1,\varepsilon \to 0}\limsup_{\tau \to 0}|(G2)^{\mathrm{bdr}}| = 0.
\end{equation*} 
Next we treat the term $(G32)^{\mathrm{bdr}}$. Again, by employing the coarea formula and letting successively $\tau \to 0$ and $\varepsilon \to 0$, we derive
\begin{equation*}
\limsup_{\varepsilon \to 0}\limsup_{\tau\to 0}|(G32)^{\mathrm{bdr}}| = 0.
\end{equation*}
\subsection{Estimate of $(H)$}
By integration by parts, we have
\begin{align*}
(H)&= \thii \frac{(\vr\uu)^\eps}{\vre}\na[(p(\vr))^\eps - p(\vre)]\dxte+ \thii \frac{(\vr\uu)^\eps}{\vre}\na p(\vre) \dxte\\
&= -\thii \di \frac{(\vr\uu)^\eps}{\vre}[(p(\vr))^\eps - p(\vre)]\dxte\\
&\quad + \thiip \frac{(\vr\uu)^\eps}{\vre}n(\theta)[(p(\vr))^\eps - p(\vre)]\dHte\\
&\quad - \thii \di(\vr\uu)^\eps \mathcal{G}(\vre)\dxte\\
&\quad + \thiip (\vr\uu)^\eps n(\theta)\mathcal{G}(\vre)\dHte\\
&=: (H1) + (H2)^{\mathrm{bdr}} + (H3) + (H4)^{\mathrm{bdr}},
\end{align*}
recalling $\mathcal G$ is defined in \eqref{G}.

By \eqref{Drho_approx}, \eqref{PG} and Gauss-Green thereom, we have 
\begin{align*}
(H3) &= \thii \partial_t \varrho^\varepsilon \mathcal{G}(\vre)\dxte \\
&= \thii \pa_t\HH(\vre) \dxte - p(1)\thii \partial_t \varrho^\varepsilon \dxte \\
&= \thii \pa_t\HH(\vre) \dxte + p(1)\thii \nabla \cdot (\varrho \uu)^\varepsilon \dxte \\
&=\thii \pa_t\HH(\vre) \dxte  + \frac{p(1)}{\varepsilon_3} \int_{\varepsilon_1}^{\varepsilon_1 + \varepsilon_3} \int_\tau^t \int_{\partial \Omega_{\varepsilon_2}}(\varrho \uu)^\varepsilon n(\theta)d\mathcal{H}^{d-1}(\theta)dsd\varepsilon_2 \\
&=: (H31) + (H32)^{\mathrm{brd}},
\end{align*}
where we recall that $\HH$ is defined in \eqref{GH}.  

The term $(H31)$ is a desired term. We now estimate the term $(H32)^{\mathrm{brd}}$ by using the coarea formula \eqref{coarea}, the fact $\mathcal L^d(\Omega \backslash\Omega_{\eps_3}) \approx \eps_3$, H\"older's inequality and Lemma \ref{Poincare},
\begin{align*}
\limsup_{\varepsilon_1,\varepsilon \to 0}\limsup_{\tau \to 0}|(H32)^{\mathrm{brd}}| &= \limsup_{\varepsilon_1,\varepsilon \to 0}\limsup_{\tau \to 0}  \frac{p(1)}{\varepsilon_3}\left|  \int_\tau^t \int_{\Omega_{\varepsilon_1} \setminus \Omega_{\varepsilon_1 + \varepsilon_3}}(\varrho \uu)^\varepsilon n(x)dxds \right| \\
&\leq  \frac{Cp(1)}{\varepsilon_3}\int_0^T\int_{\Omega \backslash\Omega_{\eps_3}}|\uu|dxds \\
&\leq \frac{C}{\varepsilon_3}(T\mathcal{L}^d(\Omega\backslash\Omega_{\eps_3}))^{\frac 12}\left(\int_0^T\int_{\Omega\backslash\Omega_{\eps_3}}|\uu|^2dxds\right)^{\frac 12}\\
&\leq Cp(1)\varepsilon_3^{\frac{1}{2}}\| \nabla \uu \|_{L^2((\Omega \setminus \Omega_{2\varepsilon_3}) \times (0,T))}.
\end{align*}
Since $\uu \in L^2(0,T;H^1(\Omega))$, it follows that 
\begin{equation*}
	\limsup_{\varepsilon_3 \to 0}\limsup_{\varepsilon_1,\varepsilon \to 0}\limsup_{\tau \to 0}|(H32)^{\mathrm{bdr}}|= 0.
\end{equation*}

The term $(H1)$ is rewritten as
\begin{align*}
(H1)&= \thii \frac{(\vr\uu)^\eps - \vre\ue}{\vre}\na[(p(\vr))^\eps - p(\vre)]\dxte\\
&\quad + \thii (\di \ue)[(p(\vr))^\eps - p(\vre)]\dxte\\
&\quad + \thiip [(p(\vr))^\eps - p(\vre)]\ue n(\theta)\dHte\\
&=: (H12) + (H13) + (H14)^{\mathrm{bdr}}.
\end{align*}
We can handle $(H12)$ and $(H13)$ using similar arguments to estimate of $(C1)$ in \eqref{X1} and \eqref{K1}, and thus
\begin{equation*}
\limsup_{\eps\to 0}\limsup_{\tau\to 0}|(H12)| = 0 \quad \text{ and } \quad \limsup_{\eps\to 0}\limsup_{\tau\to 0}|(H13)| = 0.
\end{equation*}

It remains to estimate the terms involving the boundary $(H14)^{\mathrm{bdr}}$, $(H2)^{\mathrm{bdr}}$ and $(H4)^{\mathrm{bdr}}$. By using the coarea formula and letting successively $\tau \to 0$ and $\varepsilon \to 0$, we assert that
\begin{align*}
\limsup_{\varepsilon \to 0}\limsup_{\tau \to 0} |(H14)^{\mathrm{bdr}}| = 0 \quad \text{and} \quad \limsup_{\varepsilon \to 0}\limsup_{\tau \to 0} |(H2)^{\mathrm{bdr}}| = 0.
\end{align*}
{\color{black} To deal with the term $(H4)^{\mathrm{bdr}}$, we use the coarea formula \eqref{coarea} to get 
\begin{align*}
	\limsup_{\varepsilon_1,\varepsilon \to 0}\limsup_{\tau \to 0}|(H4)^{\mathrm{bdr}}| &= \limsup_{\varepsilon_1,\varepsilon \to 0}\limsup_{\tau \to 0}\frac{1}{\eps_3}\left|\int_\tau^t\int_{\Omega_{\eps_1}\backslash\Omega_{\eps_1+\eps_3}}(\vr\uu)^\eps n(x)\mathcal{G}(\vre)dxds\right| \\
	& \leq \frac{C}{\varepsilon_3}\int_0^T\int_{\Omega \backslash\Omega_{\eps_3}}|\uu|dxds.
\end{align*}
Then by using an argument similar to the one used to estimate the term $(H32)^{\mathrm{brd}}$, we obtain 
\begin{equation*}
	\limsup_{\varepsilon_3 \to 0}\limsup_{\varepsilon_1,\varepsilon \to 0}\limsup_{\tau \to 0}|(H4)^{\mathrm{bdr}}|= 0.
\end{equation*}
}

\subsection{Estimate of $(K)$}
By rewriting $(K)$ as
\begin{align*}
(K)&= -\frac{1}{\eps_3}\int_{\eps_1}^{\eps_1+\eps_3}\int_{\tau}^{t}\int_{\Omega_{\eps_2}} \nabla \cdot (\nu(\varrho) \mathbb{D}\uu)^\varepsilon \frac{(\varrho \uu)^\varepsilon - \vre\ue}{\varrho^\varepsilon} dxdsd\eps_2\\
&\quad - \thii \di(\nu(\vr)\mathbb D\uu)^\eps \ue \dxte\\
&=: (K1) - \thiip (\nu(\vr)\mathbb D\uu)^\eps \ue n(\theta)\dHte\\
&\quad + \thii (\nu(\vr)\mathbb D\uu)^\eps \na \ue \dxte\\
&=: (K1) + (K2)^{\mathrm{bdr}} + (K3).
\end{align*}
Estimate $(K1)$ using arguments similar to \eqref{D0-1} and \eqref{for-ex1} we have
\begin{equation*}
	\limsup_{\eps\to 0}\limsup_{\tau \to 0}|(K1)| = 0.
\end{equation*}
The boundary term is computed using coarea formula \eqref{coarea}  and Lemma \ref{Poincare} as follows
\begin{equation}\label{H2}
\begin{aligned}
\limsup_{\varepsilon_1,\varepsilon \to 0}\limsup_{\tau \to 0}|(K2)^{\mathrm{bdr}}|&= \limsup_{\varepsilon_1,\varepsilon \to 0}\limsup_{\tau \to 0}  \left|\frac{1}{\eps_3}\int_\tau^t\int_{\Omega_{\varepsilon_1}\backslash\Omega_{\varepsilon_1 + \eps_3}}(\nu(\vr)\mathbb D\uu)^\eps \ue\, n(x)dxds\right|\\ 
&\leq  C\frac{1}{\eps_3} 
\| \nabla \uu \|_{L^2( (\Omega\backslash\Omega_{\eps_3}) \times (0,T))}  \| \uu \|_{L^2( (\Omega\backslash\Omega_{\eps_3}) \times (0,T))}\\ 
&\leq  C\| \nabla \uu \|_{L^2( (\Omega\backslash\Omega_{2\eps_3}) \times (0,T))}^2.
\end{aligned}
\end{equation}
Due to the assumption $\uu \in L^2(0,T;H^1(\Omega))$, by letting $\varepsilon_3 \to 0$, we get
\begin{equation*}
	\limsup_{\varepsilon_3 \to 0}\limsup_{\varepsilon_1,\varepsilon \to 0}\limsup_{\tau \to 0}|(K2)^{\mathrm{bdr}}| = 0.
\end{equation*}
\subsection{Estimate of $(L)$}
To deal with $(L)$ we estimate similarly to the estimate of $(E)$ in Subsection \ref{EstimateE}. It only remains to control the extra boundary terms
\begin{equation*}
(L)^{\mathrm{bdr}} =  - \thiip (\mu(\vr)\di \uu)^\eps \ue n(\theta)\dHte.
\end{equation*}
This term can be estimated by using an argument similar to the one leading to \eqref{H2}. Therefore we obtain
\begin{equation*}
\limsup_{\varepsilon_3 \to 0}\limsup_{\varepsilon_1, \varepsilon\to 0} \limsup_{\tau \to 0}|(L)^{\mathrm{bdr}}| = 0.
\end{equation*}

\subsection{Conclusion of the Proof of Theorem \ref{cNS-domain}}
Collecting the estimates for $(F)$, $(G)$, $(H)$, $(K)$ and $(L)$ and using similar arguments to subsection \ref{sub:conclusion1} we obtain the desired energy equality \eqref{eq_cNSd}. \qed

%\section{Proof of Theorem \ref{cNS-normal-both}}
%The proof of Theorem \ref{cNS-normal-both} follows closely from that of Theorems \ref{cNS-torus} and \ref{cNS-domain} except we have to take extra care of the terms $-2\nu \Delta \uu$ and $-\lambda \na(\di \uu)$.
%
%\medskip
%In the case $\Omega = \T^d$, by multiplying the smoothed version of \eqref{cNS-normal},
%\begin{equation*}
%	\pa_t(\vr\uu)^\eps + \di (\vr\uu\otimes \uu)^\eps + \na(\vr^\gamma)^\eps - 2\nu \Delta \ue - \lambda \na(\di \uu)^\eps = 0
%\end{equation*}
%by $(\vr\uu)^\eps/\vre$ we can proceed similarly to the proof of Theorem \ref{cNS-torus} except for the extra terms
%
%\medskip
%In the case $\Omega \subset \R^d$ is a bounded domain with $C^2$ boundary, by proceeding as in section \ref{sec:thm2} and using the arguments dealing with \eqref{extra1} and \eqref{extra2}, we are left to estimate only the extra boundary terms
%\begin{equation*}
%(BEX1) = -2\nu\thiip \na \ue \ue n(\theta)\dHte,
%\end{equation*}
%and
%\begin{equation*}
%(BEX2) = -\lambda \thiip (\di \uu)^\eps \ue n(\theta)\dHte.
%\end{equation*}
%Both of these terms can be estimated by using an argument similar to the one leading to \eqref{H2}. Therefore we obtain
%\begin{equation*}
%\limsup_{\varepsilon_3 \to 0}\limsup_{\varepsilon_1, \varepsilon\to 0} \limsup_{\tau \to 0}|(BEX1)| = 0 \quad \text{ and } \quad \limsup_{\varepsilon_3 \to 0}\limsup_{\varepsilon_1, \varepsilon\to 0} \limsup_{\tau \to 0}|(BEX2)| = 0.
%\end{equation*}
%Thus the proof of Theorem \ref{cNS-normal-both} is complete. \qed
\section{Proof of Theorem \ref{icNS-both}}
Thanks to the proof of Theorems \ref{cNS-torus} and \ref{cNS-domain} we only need to take care of the terms concerning the scalar pressure.

\medskip
In case $\Omega = \T^d$, we multiply
\begin{equation*}
	\pa_t(\vr\uu)^\eps + \di(\vr\uu\otimes\uu)^\eps + \na P^\eps - \di(\nu(\vr)\mathbb D\uu)^\eps = 0
\end{equation*}
by $(\vr\uu)^\eps/\vre$. Then we only need to take care of the following term (since other terms can be estimated similarly as in the previous Theorems)
\begin{align*}
	\int_{\tau}^t\int_{\T^d}\na P^\eps \frac{(\vr\uu)^\eps}{\vre}dxds &= \int_{\tau}^t\int_{\T^d}\na P^\eps\frac{(\vr\uu)^\eps - \vre\ue}{\vre}dxds + \int_{\tau}^t\int_{\T^d}\na P^\eps \ue dxds\\
	&= \int_{\tau}^t\int_{\T^d}\na P^\eps\frac{(\vr\uu)^\eps - \vre\ue}{\vre}dxds.
\end{align*}
At the last step we have made used of the free divergence condition. Using H\"older's inequality and Lemma \ref{lemma1} (ii), we have
\begin{equation}\label{K0}
\begin{aligned}
	\left|\int_{\tau}^t\int_{\T^d}\na P^\eps\frac{(\vr\uu)^\eps - \vre\ue}{\vre}dxds\right|&\leq \| \nabla P^\varepsilon \|_{L^2(\T^d \times(0,T) )} \| \frac{  (\vr\uu)^\eps - \vre \ue }{\varrho^\varepsilon} \|_{L^2(\T^d \times (0,T))}   \\
	&\leq C\|P\|_{L^2(\T^d \times (0,T))}\eps^{-1}\|(\vr\uu)^\eps - \vre \ue\|_{L^2(\T^d \times (0,T))}.
\end{aligned}
\end{equation}
By assumption $P\in L^2(\T^d \times (0,T))$ and Lemma \ref{keyla} we obtain the desired limit
\begin{equation*}
	\limsup_{\eps\to 0}\limsup_{\tau\to 0}\left|\int_{\tau}^t\int_{\T^d}\na P^\eps \frac{(\vr\uu)^\eps}{\vre}dxds\right| = 0
\end{equation*}
and therefore finish the proof in the case $\Omega = \T^d$.

\medskip
In case $\Omega$ is a bounded domain, we need to take care of the boundary. Similar to the case of a torus, we only need to deal with the term
\begin{align*}
	&\thii\na P^\eps \frac{(\vr\uu)^\eps}{\vre}dxds\\
	&= \thii\na P^\eps\frac{(\vr\uu)^\eps - \vre\ue}{\vre}dxds + \thii\na P^\eps \ue dxds\\
	&= \thii\na P^\eps\frac{(\vr\uu)^\eps - \vre\ue}{\vre}dxds + \thiip P^\eps\ue n(\theta)\dHte\\
	&=: (J1) + (J2)^{\mathrm{bdr}}.
\end{align*}
The term $(J1)$ is estimated exactly as in \eqref{K0} and therefore,
\begin{equation*}
	\limsup_{\varepsilon \to 0}\limsup_{\tau \to 0}|(J1)| = 0.
\end{equation*}
For $(J2)^{\mathrm{bdr}}$ we use the coarea formula, H\"older's inequality and Lemma \ref{Poincare} to obtain
\begin{align*}
\limsup_{\varepsilon_1,\varepsilon \to 0} \limsup_{\tau \to 0}|(J2)^{\mathrm{bdr}}| &=\limsup_{\varepsilon_1,\varepsilon \to 0}\limsup_{\tau \to 0} 
\left|  \frac{1}{\varepsilon_3} \int_{\tau}^t \int_{\Omega_{\varepsilon_1} \setminus \Omega_{\varepsilon_1 +\varepsilon_3}}  P^\varepsilon \uu^\varepsilon n(x) dx ds \right| \\
&\leq \frac{1}{\varepsilon_3}  \| P \|_{L^2((\Omega \backslash\Omega_{\eps_3}) \times (0,T))} \| \uu \|_{L^2((\Omega \backslash\Omega_{\eps_3}) \times (0,T))} \\
&\leq C   \| P \|_{L^2((\Omega \backslash\Omega_{\eps_3}) \times (0,T))} \| \nabla \uu \|_{L^2((\Omega \backslash\Omega_{2\eps_3}) \times (0,T))}. 
\end{align*}
Since $P \in L^2(\Omega \times (0,T))$ and $\uu \in L^2(0,T;H^1(\Omega))$,  by letting $\varepsilon_3 \to 0$, we derive
\begin{equation*}
	\limsup_{\varepsilon_3\to 0} \limsup_{\varepsilon_1, \varepsilon \to 0}\limsup_{\tau \to 0}|(J2)^{\mathrm{bdr}}| = 0.
\end{equation*}
Thus the proof of Theorem \ref{icNS-both} is complete. \qed

\section{Proof of Theorem \ref{thm:org_NS}}
The proof of Theorem \ref{thm:org_NS} follows exactly from that of Theorem \ref{icNS-both} except for the term relating to the pressure. However, since in this case we can take $\vr \equiv 1$ and therefore $(J1) = 0$ trivially. From the estimate of $(J2)^{\mathrm{bdr}}$ 
\begin{equation*}
\limsup_{\eps_1,\eps\to 0}\limsup_{\tau\to 0}|(J2)^{\mathrm{bdr}}| \leq C\|P\|_{L^2(\Omega\backslash\Omega_{\eps_3}\times(0,T))}\|\na \uu\|_{L^2(\Omega\backslash\Omega_{2\eps_3}\times(0,T))}
\end{equation*}
we use $\uu \in L^2(0,T;H^1(\Omega))$ and the assumption \eqref{pressure} to conclude $$\limsup_{\varepsilon_3\to 0} \limsup_{\varepsilon_1, \varepsilon \to 0}\limsup_{\tau \to 0}|(J2)^{\mathrm{bdr}}| = 0$$ and therefore the proof of Theorem \ref{thm:org_NS} is finished. \qed

\medskip
 \par{\bf Acknowledgements:} 
We would like to thank the referees for their comments, which help to improve the paper.
 
Phuoc-Tai Nguyen was supported by Czech Science Foundation, project
GJ19-14413Y. Bao Quoc Tang was supported by the International Training Program IGDK 1754 and NAWI Graz.

\end{document}